\documentclass[a4paper,10pt]{article}

\usepackage{amsmath}
\usepackage{german, ngerman}
\usepackage[german]{babel}
\usepackage[latin1]{inputenc}
\usepackage{graphics}
\usepackage{verbatim}
\usepackage{listings}
\usepackage{dsfont}
\usepackage{amsmath}
\usepackage{amsfonts}
\usepackage{amssymb}
\usepackage{graphicx}
\usepackage{geometry}
\usepackage{graphics}
\usepackage{url}
\usepackage{framed}
\usepackage{color}
\usepackage{textcomp}

\newtheorem{set2}{Satz}[section]
\newtheorem{theorem}[set2]{Theorem}

\newtheorem{definition}[set2]{Definition}

\newtheorem{lemma}[set2]{Lemma}
\newtheorem{notation[set2]}{Notation}

\newtheorem{remark}[set2]{Remark}

\newcommand{\ep}{\hfill{$\square$}}
\newenvironment{proof}[1][Proof]{\textbf{#1.} }{\
\\}

\def\XXint#1#2#3{{\setbox0=\hbox{$#1{#2#3}{\int}$}
\vcenter{\hbox{$#2#3$}}\kern-.5\wd0}}

\newcommand{\Hb}{H^2(\Omega;\R^n)}
\newcommand{\Wa}{W^{1,p}(\Omega)}

\newcommand{\bl}{\left(}
\newcommand{\br}{\right)}
\newcommand{\e}{\epsilon}
	
\newcommand{\di}{\,\mathrm{div}}
\newcommand{\R}{\mathbb{R}}
\newcommand{\N}{\mathbb{N}}
\newcommand{\C}{\mathcal}
\newcommand{\ol}{\overline}
\newcommand{\dx}{\,\mathrm dx}

\newcommand{\dt}{\,\mathrm dt}
\newcommand{\ds}{\,\mathrm ds}

\newcommand{\dxt}{\,\mathrm dx\,\mathrm dt}
\newcommand{\dxs}{\,\mathrm dx\,\mathrm ds}

\usepackage{bm}
\allowdisplaybreaks

\date{19.12.2012}

\begin{document}
\selectlanguage{english}
$\phantom{w}$
\begin{center}
\LARGE
Existence of weak solutions for
a PDE system describing phase separation and damage processes including inertial effects\footnote{This project is supported by ECMath in Berlin (Germany).}
\end{center}
\begin{center}
	Christian Heinemann$^2$, Christiane Kraus\footnote{Weierstrass Institute for Applied Analysis and Stochastics (WIAS), Mohrenstr. 39, 10117 Berlin
	\\
	E-mail: \texttt{christian.heinemann@wias-berlin.de} and
\texttt{christiane.kraus@wias-berlin.de}}
\end{center}
\begin{center}
{\it Dedicated to J\"urgen Sprekels on the occasion of his 65th birthday}
\end{center}
\begin{abstract}
	In this paper, we  consider a coupled PDE system describing phase separation and damage phenomena in elastically stressed alloys
	in the presence of inertial effects.
	The material is considered on a bounded Lipschitz domain with mixed
	boundary conditions for the displacement variable.
        The main aim of this work is to establish existence of weak solutions for the
	introduced hyperbolic-parabolic system. To this end,
	we first adopt the notion of weak solutions introduced
	in \cite{WIAS1520}. Then we
	prove existence of weak solutions by means of regularization, time-discretization
	and different variational techniques.
\end{abstract}

\noindent
{\it AMS Subject classifications}
35L20,   
35L51, 
35K86, 
35K55, 
49J40, 
49S05, 
74A45,	
74G25, 
34A12,  
82B26,  
82C26,  
35K92,   
35K35.  
\\[2mm]
{\it Keywords:} {Cahn-Hilliard system, phase separation, hyperbolic-parabolic systems, doubly nonlinear
  differential inclusions, existence results,
  energetic solutions, weak solutions,  linear elasticity, rate-dependent systems.  \\[2mm]
}


\section{Introduction}

	In micro-electronic materials such as solder alloys, different physical processes are shaping the micro-structure.
	For a realistic description of these structures, phase separation, coarsening and
	elasticity as well as damage phenomena have to be taken into account.
	A fully coupled system has been originally studied in \cite{WIAS1520}
	and further developed in \cite{WIAS1569} allowing, for instance, inhomogeneous elastic
	energy densities. The corresponding degenerating case has been analyzed in \cite{WIASxxxx}.
	To the authors' best knowledge, before these works, phase separation
	and damage processes have only been investigated independently of each
	other in the mathematical literature.

Phase separation and coarsening phenomena are usually described by
phase--field models of Cahn-Hilliard type. The evolution is
modeled by  a parabolic diffusion equation for the phase fractions.
To include elastic effects, resulting from stresses caused by different
elastic properties of the phases, Cahn-Hilliard systems are coupled
with an elliptic equation in the case of a quasi-static balance of
forces. Such coupled Cahn-Hilliard systems with elasticity are also called
Cahn-Larch\'e systems.  Since in general the mobility,
stiffness and surface tension coefficients depend on the phases (see
for instance \cite{BDM07} and \cite{BDDM09} for the explicit structure deduced by the embedded atom
method), the mathematical analysis of the coupled problem is very
complex. Existence results were derived for special cases in
\cite{Carrive00, Gar00, Pawlow}
(constant mobility, stiffness and surface tension coefficients),
in \cite{Bonetti02} (concentration dependent mobility,  two
space dimensions), \cite{SP12, SP13}  (concentration dependent surface
tension and nonlinear diffusion) and in \cite{Pawlow08} in an abstract measure-valued
setting (concentration dependent mobility and surface tension tensors).

Damage behavior, however, originates from breaking atomic links in the material
from a microscopic point of view
whereas a macroscopic theory may specify damage in the isotropic case by a
scalar-valued variable related to the proportion of damaged bonds in the
micro-structure of the material with respect to the undamaged ones.
According to the latter perspective, phase-field models are quite
common to model smooth transitions between damaged and undamaged
material states. Such phase-field models have been mainly
investigated for incomplete damage which means that damaged material cannot
loose all its elastic energy.

Existence and uniqueness results for damage models of viscoelastic materials are proven in
\cite{BSS05} for scalar-valued displacements. Higher dimensional damage models are
analytically investigated in \cite{BS04, Mielke06, MT10, KRZ11, RR12} and, there, existence
and regularity properties are shown. A coupled system describing
incomplete damage, linear elasticity and phase separation appeared
in \cite{WIAS1520, WIAS1569}. There, existence of weak solutions has
been proven under mild assumptions, where, for instance, the stiffness
tensor may be material-dependent and the chemical free energy may be
of polynomial or logarithmic type.
All these works are based on the gradient-of-damage model proposed by
Fr\'emond and Nedjar \cite{FN96} (see also \cite{Fre02}) which
describes damage as a result from microscopic movements in the solid.
The distinction between a balance law for the microscopic forces and
constitutive relations of the material yield a satisfying derivation
of an evolution law for the damage propagation from the physical point of view.
In particular, the gradient of the damage variable enters the resulting equation
and serves as a regularization term for the mathematical analysis as well as it ensures the structural size effect.
Internal constraints are ensured by the presence of non-smooth operators
(subdifferential operators) in the evolution system. Hence, in the case that the evolution of the damage is assumed to be uni-directional, i.e. the
damage process is irreversible, the microforce balance law becomes a doubly-nonlinear differential inclusion.

	The main aim of this paper is to generalize the results for
        hyperbolic-parabolic damage systems introduced
        in \cite{HK13_1} to coupled phase-field systems describing {\it phase
        separation} and {\it damage processes}
	in the presence of inertial terms with mixed boundary conditions on
        {\it non-smooth (Lipschitz) domains}. The novelty of this contribution is to obtain
        existence results for {\it phase separation} with elasticity including
        {\it inertial effects} and damage processes on Lipschitz domains.
        We first utilize and adjust the notion of weak solutions introduced
	in \cite{WIAS1520}. Then, we prove existence of weak solutions by means of regularization, time-discretization
	and different variational techniques. To this end, an energy estimate
        has, for instance, to be established and several convergence
        properties are shown.

\subsection{Energies and evolutionary equations}
Here, we qualify our model formally and postpone a rigorous
treatment to Section \ref{section:mainProofs}.
The presented model is based on two functionals, i.e.~a
generalized Ginzburg-Landau free energy functional $\mathcal E$ and a damage
pseudo-dissipation potential $\mathcal R$
(in the sense by Moreau). The free energy density $\varphi$ of
the system is given by
\begin{equation}
  \label{eq:free_energy}
  \varphi(\varepsilon(u),c,\nabla c,z,\nabla z):=
\frac{1}{p}|\nabla z|^p +
\frac{1}{2} |\nabla
  c |^2  +W(c, \varepsilon(u),z) + f(z) + \Psi(c) ,
\end{equation}
where the gradient terms  penalize spatial changes of the
variables $c$ and $z$.
$W$ denotes the elastically stored energy density
accounting for elastic deformations and damage effects, $f$ is the damage
dependent potential and $\Psi$ stands for  the
chemical energy density.

	The overall free energy ${\mathcal E}$ of Ginzburg-Landau type
                has the following structure:
		\begin{equation}
		\label{eqn:EnergyTyp1}
			\begin{split}
				&\mathcal E(u,c,z):=\int_\Omega\Big(\varphi(\varepsilon(u),c,\nabla c,z,\nabla z)+I_{[0,\infty)}(z)\Big)\,\mathrm dx.
			\end{split}
		\end{equation}
		In this context, $I_{[0,\infty)}$ signifies the indicator function of
                the subset $[0,\infty)\subseteq\mathbb R$, i.e. $ I_{[0,\infty)}(x)=0$ for $x \in
                    [0,\infty)$ and  $ I_{[0,\infty)}(x) =\infty$ for $x<0$.
                    We assume that the energy dissipation
                for the damage process is triggered by a rate-dependent
		dissipation potential ${\mathcal R}$ of the form
		\begin{equation}
		\label{eqn:EnergyTyp2}
			\begin{split}
				&{\mathcal R}(\dot z):=\int_\Omega\Big(\frac 12|\dot z|^2+I_{(-\infty,0]}(\dot z)\Big)\,\mathrm dx.
			\end{split}
		\end{equation}
			
		The governing evolutionary equations for a system state $q=(u,c,z)$ can be expressed by virtue of the
		functionals \eqref{eqn:EnergyTyp1} and \eqref{eqn:EnergyTyp2}.
		More precisely, the evolution is driven by the
		following hyperbolic-parabolic system of differential equations
		and differential inclusions:
	\begin{subequations}
	\label{eqn:PDE}
	\begin{align}
		\label{eqn:pde1}
		&\textit{diffusion:}	&&c_t=\di(m(c,z)\nabla\mu), \\
		\label{eqn:pde2}
			&&&\mu=-\Delta c+W_{,c}(c,\e(u),z)+\Psi'(c),\\
		\label{eqn:pde3} \quad
		&\textit{balance of forces:}	&&u_{tt}-\di\bl W_{,e}(c,\e(u),z)\br=l,\\
		\label{eqn:pde4}
		&\textit{damage evolution:}	&&0\in\partial_z\mathcal
			E(u,c,z)+\partial_{\dot z}\mathcal R(\partial_t z)  \quad \textit{
			or equivalently }\\
			&&&z_t -\Delta_p z+W_{,z}(c,\e(u),z)+f'(z)+\xi+\varphi=0,\\
		\label{eqn:pde5}
			&&&\xi\in\partial I_{[0,\infty)}(z),\\
		\label{eqn:pde6}
			&&&\varphi\in\partial I_{(-\infty,0]}(z_t).
	\end{align}
	\end{subequations}
        	The Cahn-Hilliard system \eqref{eqn:pde1}-\eqref{eqn:pde2} describes phase separation phenomena in alloys, the
	hyperbolic equation \eqref{eqn:pde3} formulates the balance of forces including inertial effects
	and the inclusion \eqref{eqn:pde4}-\eqref{eqn:pde6} is an evolution law for the damage processes.
	The sub-gradients correspond to the constraints that the
	damage is non-negative and irreversible.	Let us note that linear contributions in $f$ model damage activation thresholds.
	
		We choose Dirichlet conditions for the displacements $u$ on a subset $\Gamma$ of the boundary
		$\partial\Omega$ with $\mathcal H^{n-1}(\Gamma)>0$.
		Let $b:[0,T]\times\Gamma\rightarrow \mathbb R^n$ be a function
                which prescribes
		the displacements on $\Gamma$ for a fixed chosen time interval $[0,T]$.
		The imposed boundary and initial conditions and constraints
                are as follows:
               	\begin{subequations}
	        \begin{align}
		&\textit{boundary displacements}:&&u=b \text{ on }{\Gamma_\mathrm{D}}\times(0,T),\\
		&\textit{initial concentration}: &&c(0)=c^0\text{ in }\Omega,\\
		&\textit{initial displacements}: &&u(0)=u^0,\;u_t(0)=v^0\text{ in }\Omega,\\
		&\textit{initial damage}:&&z(0)=z^0\text{ in }\Omega.
	\end{align}
	\end{subequations}
		Moreover, we use natural boundary conditions for the remaining variables
		on (parts of) the boundary:
                \begin{subequations}
	\label{eqn:PDEIBC}
	\begin{flalign}
                &\qquad W_{,e}(c,\e(u),z)\cdot\nu=0&&\hspace*{-16em}\text{ on }\Gamma_\mathrm{N}\times(0,T),\\
		&\qquad \nabla c\cdot\nu=\nabla
		z\cdot\nu=m(c,z)\nabla \mu\cdot\nu=0&&\hspace*{-16em}\text{ on
		}\partial\Omega, \label{eqn:boundary2}
\end{flalign}
\end{subequations}
where $\nu$ stands for the outer unit normal to $\partial\Omega$.
		
		We like to mention that mass conservation of the system
                follows from the diffusion
                equation \eqref{eqn:pde1}
		and \eqref{eqn:boundary2}, i.e.
                $$ \int_\Omega c(t)-c^0\,\mathrm dx=0\text{ for all }t\in[0,T].$$

	In the next section, we state the precise assumptions that are needed for a rigorous analysis.
	Section \ref{section:existence} presents the main results.
	We give a notion of weak solutions evolved from \cite{HK13_1} and state the existence theorem in Subsection \ref{section:weakNotion}.
	Since the proof is based on regularization techniques, we also give the weak notion and the associated existence result for the regularized system
	in Subsection \ref{section:H2reg}.
	In the main part, Section \ref{section:mainProofs}, the existence proof is carried out first for the regularized case and
	then for the limiting case.

%
\section{Notation and assumptions}
\label{section:assumptions}

	Throughout this work, let $p>n$ be a constant and let $\Omega\subseteq\R^n$ ($n=1,2,3$) be a bounded Lipschitz domain.
	For the Dirichlet boundary $\Gamma_\mathrm{D}$ and the Neumann boundary $\Gamma_\mathrm{N}$ of $\partial\Omega$,
	we adopt the assumptions from \cite{Ber11}, i.e., $\Gamma_\mathrm{D}$ and $\Gamma_\mathrm{N}$ are non-empty and relatively open sets in
	$\partial\Omega$ with finitely many path-connected components such that $\Gamma_\mathrm{D}\cap \Gamma_\mathrm{N}=\emptyset$
	and $\ol{\Gamma_\mathrm{D}}\cup \ol{\Gamma_\mathrm{N}}=\partial\Omega$.
	
	The considered time interval is denoted
	by $[0,T]$ and  $\Omega_t:=\Omega \times [0,t]$ for $t \in [0,T]$. The
	partial derivative of a function $h$ with respect to a variable $s$ is
	abbreviated by $h_{,s}$.
	The set $\{v>0\}$ for a function $v\in W^{1,p}(\Omega)$ has to be read as $\{x\in\ol\Omega\,|\,v(x)>0\}$ by employing the embedding
	$W^{1,p}(\Omega)\hookrightarrow\C C(\ol\Omega)$ (because $p>n$).
	
	The elastic energy density $W$ is assumed to be of the form
	\begin{align}
	\label{eqn:defW}
		W(c,e,z)=\frac 12 \mathbf C(z)(e-e^*(c)):(e-e^*(c)),
	\end{align}
	where $e^*$ denotes the eigenstrain and $\mathbf C$ the material
	stiffness tensor which depends on the damage variable. For $e^*$,
	we assume the linear relation $e^*(c)= c \, \hat{e}$ with $\hat{e}\in\R_\mathrm{sym}^{n\times n}$ (Vegard's law).
	We choose the stiffness tensor function $\mathbf C \in \C C^1([0,1];\C L_\text{sym}(\R^{n\times n}))$,
	where $\C L_\text{sym}(\R^{n\times n})$ denotes the linear mappings from $\R^{n\times n}$ into $\R^{n\times n}$ which are
	symmetric.
	We also assume the properties
	\begin{align}
	\label{eqn:assumptionC}
		\mathbf C(z)e:e\geq\eta|e|^2,\qquad \mathbf C'(z)e:e\geq 0
	\end{align}
	for all $e\in\R_\mathrm{sym}^{n\times n}$, $z\in[0,1]$ and a constant $\eta>0$ independent of $e$ and $z$.

	
	Furthermore, we choose the mobility $m\in C(\R\times[0,1];\R^+)$ and
	suppose that the chemical energy density $\Psi\in \C C^1(\R)$ can be decomposed into
	$$ \Psi(c)= \Psi_1(c) + \Psi_2(c) \quad \text{for }c\in\R,$$
	where $\Psi_1,\Psi_2\in\C C^1(\R)$ with $\Psi_1$ convex and $\Psi_1\ge 0$. 
	
	In addition, we assume 
	the following growth conditions:
	\begin{subequations}
	\begin{align}
	\label{eqn:assumption_psi}
		|\Psi'(c)|&\leq C(1+|c|^{2^\star/2}), \\
	\label{eqn:assumption_psi2}
		|\Psi_2'(c)|& \le C( |c| +1)
	\end{align}
	\end{subequations}
	for all $c \in \R$.
	Moreover, the mobility function should satisfy
	\begin{align}
	\label{eqn:assumption_m}
		C_1 \leq m(c,z) &\leq C_2
	\end{align}
	for all $c \in \R$, $z \in [0,1]$.
	Here, $C_1,C_2>0$ denote constants independent of $c$ and $z$, and
	$2^\star$ is the Sobolev critical exponent.
	
	The damage dependent potential $f$ entering equation \eqref{eqn:pde4}
	is assumed to be a function of $\C C^1([0,1];\R^+)$.
	
\section{Main results}
\label{section:existence}
\subsection{Notion of weak solutions and existence results}
\label{section:weakNotion}
	In what follows we define for $k \ge 1$ the spaces	
	\begin{align*}
		W_+^{k,p}(\Omega)&:=\big\{u\in W^{k,p}(\Omega)\,|\,u\geq 0\text{ a.e. in }\Omega\big\},\\
		W_-^{k,p}(\Omega)&:=\big\{u\in W^{k,p}(\Omega)\,|\,u\leq 0\text{ a.e. in }\Omega\big\},\\
		H_{\Gamma_\mathrm{D}}^{k}(\Omega)&:=\big\{u\in H^{k}(\Omega)\,|\,u= 0\text{
	on }{\Gamma_\mathrm{D}}\text{ in the sense of traces}\big\} .
	\end{align*}
	
	Let the following initial-boundary data and volume forces be given:
	\begin{align*}
		&\textit{boundary data:}
			&&b\in H^{1}(0,T;H^2(\Omega;\R^n)\cap W^{2,1}(0,T;L^2(\Omega;\R^n)),\\
		&\textit{initial values:}
			&&c^0\in H^1(\Omega),\;u^0\in H^1(\Omega;\R^n),\;v^0\in L^2(\Omega;\R^n),\\
			&&&z^0\in W^{1,p}(\Omega) \text{ with }0\leq z^0\leq 1\text{ a.e. in }\Omega,\\
		&\textit{external volume forces:}
			&&l\in L^2(0,T;L^2(\Omega;\R^n)).
	\end{align*}
	A weak formulation of system \eqref{eqn:PDE}-\eqref{eqn:PDEIBC} is given in the following definition.
	\begin{definition}[Weak solution]
	\label{def:weakSolution}
	A weak solution of the PDE system \eqref{eqn:PDE}-\eqref{eqn:PDEIBC} for the data $(l,b,c^0,u^0,v^0,z^0)$ is a 5-tuple $(c,u,z,\mu,\xi)$ satisfying
	the following properties:
	\begin{itemize}
		\item
			spaces:
			\begin{flalign*}
				&\qquad c\in L^\infty(0,T;H^1(\Omega))\cap H^1(0,T;(H^1(\Omega))^*),\\
				&\qquad u\in L^\infty(0,T;H^1(\Omega;\R^n))\cap W^{1,\infty}(0,T;L^2(\Omega;\R^n))\cap H^2(0,T;(H_{\Gamma_\mathrm{D}}^1(\Omega;\R^n))^*)\\
				&\qquad\qquad\text{with }u=b\text{ on }{\Gamma_\mathrm{D}}\times(0,T),\;u(0)=u^0\text{ a.e. in }\Omega,\;\partial_t u(0)=v^0\text{ a.e. in }\Omega,&\\
				&\qquad z\in L^\infty(0,T;W^{1,p}(\Omega))\cap H^1(0,T;L^2(\Omega))&\\
				&\qquad\qquad\text{with }z(0)=z^0\text{ in }\Omega,\;z\geq 0\text{ a.e. in }\Omega_T,\;\partial_t z\leq 0\text{ a.e. in }\Omega_T,&\\
                                &\qquad  \mu \in L^2(0,T;H^1(\Omega)),		\\
				&\qquad \xi\in L^\infty(0,T;L^1(\Omega)).&
			\end{flalign*}
		\item
			for all $\zeta \in L^2(0,T;H^1(\Omega))\cap H^1(0,T;L^2(\Omega))$ with $\zeta(T)=0$:
			\begin{align}
			\label{eqn:weak1}
				\int_{\Omega_T}(c-c^0)\partial_t\zeta\dxt=\int_{\Omega_T}m(c,z)\nabla\mu\cdot\nabla\zeta\dxt
			\end{align}
		\item
			for all $\zeta\in L^2(0,T;H^1(\Omega))$ and for a.e. $t\in(0,T)$:
			\begin{align}
			\label{eqn:weak2}
				\int_{\Omega}\mu \, \zeta\dx=\int_{\Omega}\big(\nabla c\cdot\nabla \zeta+W_{,c}(c,\e(u),z)\zeta+\Psi'(c)\zeta\big)\dx
			\end{align}
		\item
			for all $\zeta\in H_{\Gamma_\mathrm{D}}^1(\Omega;\R^n)$ and for a.e. $t\in(0,T)$:
			\begin{align}
			\label{eqn:weak3}
				\langle \partial_{tt} u,\zeta\rangle_{H^1}+\int_\Omega W_{,e}(c,\e(u),z):\e(\zeta)\dx=\int_\Omega l\cdot\zeta\dx
			\end{align}
		\item
			for all $\zeta\in W_-^{1,p}(\Omega)$ and for a.e. $t\in(0,T)$:
			\begin{align}
			\label{eqn:weak4}
				0\leq\int_\Omega\bl|\nabla z|^{p-2}\nabla z\cdot\nabla\zeta+(W_{,z}(c,\e(u),z)+f'(z)+\partial_t z+\xi)\zeta\br\dx
			\end{align}
		\item
			for all $\zeta\in L_+^{\infty}(\Omega)$ and for a.e. $t\in(0,T)$:
			\begin{align}
			\label{eqn:weak5}
				0\geq\int_\Omega\xi(\zeta-z)\dx
			\end{align}
		\item
			total energy inequality for a.e. $t\in(0,T)$:
			\begin{align}
			\label{eqn:weak6}
				\C E(t)+\C K(t)+\C D(0,t)\leq \C E(0)+\C K(0)+\C W_\mathrm{ext}(0,t)
			\end{align}
			with
			\begin{align*}
				&\textit{free energy:}
					&&\C E(t):=\int_\Omega\bl\frac 1p|\nabla z(t)|^p+\frac 12|\nabla c(t)|^2+W(c(t),\e(u(t)),z(t))\br\dx\\
					&&&\qquad\quad+\int_\Omega\big(f(z(t))+\Psi(c(t))\big)\dx,\\
				&\textit{kinetic energy:}
					&&\C K(t):=\int_\Omega\frac 12|\partial_t u(t)|^2\dx,\\
				&\textit{dissipation:}
					&&\C D(0,t):=\int_{\Omega_t}\big(|\partial_t z|^2+m(c,z)|\nabla\mu|^2\big)\dxs,\\
				&\textit{external work:}
					&&\C W_\mathrm{ext}(0,t):=
					\int_{\Omega_t} W_{,e}(c,\e(u),z):\e(\partial_t b)\dxs\\
					&&&\qquad\qquad\quad\;-\int_{\Omega_t}\partial_{t}u\cdot\partial_{tt} b\dxs
					+\int_{\Omega_t}l\cdot (\partial_t u-\partial_t b)\dxs\\
					&&&\qquad\qquad\quad\;-\int_\Omega v^0\cdot \partial_t b^0\dx
					+\int_\Omega \partial_t u(t)\cdot \partial_t b(t)\dx.
			\end{align*}
	\end{itemize}
	\end{definition}
	
	\begin{remark}
		Let $(c,u,z,\mu,\xi)$ be a weak solution. Furthermore, if additionally
		\begin{align*}
			&c\in H^1(0,T;H^1(\Omega)),\quad u\in H^1(0,T;H^1(\Omega;\R^n)),\quad
			z\in H^1(0,T;W^{1,p}(\Omega)),
		\end{align*}
		then for a.e. $t\in(0,T)$
		\begin{align*}
			\qquad z_t&-\Delta_p z+W_{,z}(c,\e(u),z)+f'(z)+\xi+\varphi=0\text{ in }\bl W^{1,p}(\Omega)\br^*,\\
			\qquad \xi&\in\partial I_{W_+^{1,p}(\Omega)}(z),\\
			\qquad \varphi&\in\partial I_{W_-^{1,p}(\Omega)}(\partial_t z).
		\end{align*}
		Moreover, the energy inequality \eqref{eqn:weak6} becomes an energy balance.
	\end{remark}
	
	The main aim of this work is to prove existence of weak solutions in the sense above.
	\begin{theorem}
	\label{theorem:mainResult}
		Let the assumptions in Section \ref{section:assumptions} be satisfied.
		To the given data $l$, $b$, $c^0$, $u^0$, $v^0$, $z^0$, there exists a weak solution of system \eqref{eqn:PDE}-\eqref{eqn:PDEIBC}
		in the sense of Definition \ref{def:weakSolution}.
	\end{theorem}
	
\subsection{Notion of weak solutions for a regularized system and existence results}
\label{section:H2reg}
	We will first study a regularized version of our phase separation-damage model. The passage to the limit is performed in Section \ref{section:limit}.
	The regularization is needed in the existence proof in the first instance to pass from the time-discrete to the time-continuous system.
	
	The regularized PDE system for $\delta>0$ is given by
	\begin{flalign*}
		\qquad c_t&=\di(m(c,z)\nabla\mu),\\
		\qquad \mu&=-\Delta c+W_{,c}(c,\e(u),z)+\Psi'(c)+\delta c_t,\\
		\qquad u_{tt}&-\di\bl W_{,e}(c,\e(u),z)\br+\delta Au=l,&\\
		\qquad z_t&-\Delta_p z+W_{,z}(c,\e(u),z)+f'(z)+\xi+\varphi=0,&\\
		\qquad \xi&\in\partial I_{[0,\infty)}(z),&\\
		\qquad \varphi&\in\partial I_{(-\infty,0]}(z_t),&
	\end{flalign*}
	where the linear operator $A:H^2(\Omega;\R^n)\to (H^2(\Omega;\R^n))^*$ is defined as
	$$
		\langle Au,v\rangle_{H^2}
		:=\int_\Omega\langle\nabla(\nabla u),\nabla(\nabla v)\rangle_{\R^{n\times n\times n}}\dx:=
			\sum_{1\leq i,j,k\leq n}\int_\Omega \frac{\mathrm d^2 u_k}{\mathrm dx_i\mathrm d x_j}\frac{\mathrm d^2 v_k}{\mathrm dx_i\mathrm d x_j}\dx.
	$$
	A weak formulation of the regularized system such as in Definition \ref{def:weakSolution} can be
	obtained with the corresponding modifications including the
	$\delta$-terms.
	\begin{definition}[Weak solution of the regularized system]
	\label{def:regWeakSolution}
	A weak solution of the regularized PDE system for the data $(l,b,c^0,u^0,v^0,z^0)$ is a 5-tuple $(c,u,z,\mu,\xi)$ satisfying
	the following properties:
	\begin{itemize}
		\item
			spaces:
			\begin{flalign*}
				&\qquad c\in L^\infty(0,T;H^1(\Omega))\cap H^1(0,T;L^2(\Omega)),\\
				&\qquad\qquad\text{with }c(0)=c^0\text{ a.e. in }\Omega,\\
				&\qquad u\in L^\infty(0,T;H^2(\Omega;\R^n))\cap W^{1,\infty}(0,T;L^2(\Omega;\R^n))\cap H^2(0,T;(H_{\Gamma_\mathrm{D}}^2(\Omega;\R^n))^*)\\
				&\qquad\qquad\text{with }u=b\text{ on }{\Gamma_\mathrm{D}}\times(0,T),\;u(0)=u^0\text{ a.e. in }\Omega,\;\partial_t u(0)=v^0\text{ a.e. in }\Omega,&\\
				&\qquad z\in L^\infty(0,T;W^{1,p}(\Omega))\cap H^1(0,T;L^2(\Omega))&\\
				&\qquad\qquad\text{with }z(0)=z^0\text{ in }\Omega,\;z\geq 0\text{ a.e. in }\Omega_T,\;\partial_t z\leq 0\text{ a.e. in }\Omega_T,&\\
				&\qquad  \mu \in L^2(0,T;H^1(\Omega)),\\
				&\qquad \xi\in L^\infty(0,T;L^1(\Omega)).&
			\end{flalign*}
		\item
			for all $\zeta\in H^1(\Omega)$ and for a.e. $t\in(0,T)$:
			\begin{align}
			\label{eqn:weak1_delta}
				\int_{\Omega_T}(\partial_t c)\,\zeta\dxt=-\int_{\Omega_T}m(c,z)\nabla\mu\cdot\nabla\zeta\dxt
			\end{align}
		\item
			for all $\zeta\in H^1(\Omega)$ and for a.e. $t\in(0,T)$:
			\begin{align}
			\label{eqn:weak2_delta}
				\int_{\Omega}\mu \, \zeta\dx=\int_{\Omega}\big(\nabla c\cdot\nabla \zeta+W_{,c}(c,\e(u),z)\zeta+\Psi'(c)\zeta
				+\delta \,(\partial_t c)\, \zeta \big)\dx
			\end{align}
		\item
			for all $\zeta\in H_{\Gamma_\mathrm{D}}^1(\Omega;\R^n)$ and for a.e. $t\in(0,T)$:
			\begin{align}
			\label{eqn:weak3_delta}
				\langle \partial_{tt}u,\zeta\rangle_{H^1}+\int_\Omega W_{,e}(c,\e(u),z):\e(\zeta)\dx
				  + \delta  \langle Au ,\zeta\rangle_{H^2}=\int_\Omega l\cdot\zeta\dx
			\end{align}
		\item
			for all $\zeta\in W_-^{1,p}(\Omega)$ and for a.e. $t\in(0,T)$:
			\begin{align}
			\label{eqn:weak4_delta}
				0\leq\int_\Omega\bl|\nabla z|^{p-2}\nabla z\cdot\nabla\zeta+(W_{,z}(c,\e(u),z)+f'(z)+\partial_t z+\xi)\zeta\br\dx
			\end{align}
		\item
			for all $\zeta\in L_+^{\infty}(\Omega)$ and for a.e. $t\in(0,T)$:
			\begin{align}
			\label{eqn:weak5_delta}
				0\geq\int_\Omega\xi(\zeta-z)\dx
			\end{align}
		\item
			total energy inequality for a.e. $t\in(0,T)$:
			\begin{align}
			\label{eqn:weak6_delta}
				\C E(t)+\C K(t)+\C D(0,t)\leq \C E(0)+\C K(0)+\C W_\mathrm{ext}(0,t)
			\end{align}
			with
			\begin{align*}
				&\textit{free energy:}
					&&\C E(t):=\int_\Omega\bl\frac 1p|\nabla z(t)|^p+\frac 12|\nabla c(t)|^2+W(c(t),\e(u(t)),z(t))\br\dx\\
					&&&\qquad\quad+\int_\Omega\big(f(z(t))+\Psi(c(t))\big)\dx
					  + \frac{\delta}{2} \langle A u_\tau(t) , u_\tau(t)\rangle_{H^2},\\
				&\textit{kinetic energy:}
					&&\C K(t):=\int_\Omega\frac 12|\partial_t u(t)|^2\dx,\\
				&\textit{dissipation:}
					&&\C D(0,t):=\int_{\Omega_t}\big(|\partial_t z|^2+ \delta|\partial_t c|^2+m(c,z)|\nabla\mu|^2\big)\dxs,\\
				&\textit{external work:}
					&&\C W_\mathrm{ext}(0,t):=
					\int_{\Omega_t} W_{,e}(c,\e(u),z):\e(\partial_t b)\dxs\\
					&&&\qquad\qquad\quad\;+\delta\int_0^t\langle A u(s),\partial_t b(s)\rangle_{H^2}\ds\\
					&&&\qquad\qquad\quad\;-\int_{\Omega_t}\partial_{t}u\cdot\partial_{tt} b\dxs
					+\int_{\Omega_t}l\cdot (\partial_t u-\partial_t b)\dxs\\
					&&&\qquad\qquad\quad\;-\int_\Omega v^0\cdot \partial_t b^0\dx
					+\int_\Omega \partial_t u(t)\cdot \partial_t b(t)\dx.
			\end{align*}
	\end{itemize}
	\end{definition}
	The proof of the main result, see Theorem \ref{theorem:mainResult}, is
	based on  the existence of weak solutions for the regularized system.
	\begin{theorem}
	\label{theorem:mainResult_delta}
		Let the assumptions in Section \ref{section:assumptions} be satisfied.
		To the given data $l$, $b$, $c^0$, $u^0$, $v^0$, $z^0$, there exists a weak solution of the regularized system
		in the sense of Definition \ref{def:regWeakSolution}.
	\end{theorem}

\section{Proof of the existence theorems}
\label{section:mainProofs}
\subsection{Existence proof for the regularized system}
	For the existence proof of the regularized system, we will use a
	semi-implicit Euler scheme solved by a recursive minimization
	procedure.

	Let $\tau>0$ denote the discretization fineness and let
	$M_\tau:=\lfloor T/\tau \rfloor$ be the number of discrete time
	points. We fix a $k\in{1,\ldots, M_\tau}$ and define the functional $\C
	F_\tau^k:H^1(\Omega)\times H^2(\Omega;\R^n)\times W^{1,p}(\Omega)\to\R$ by
	\begin{align*}
		\C F_\tau^k(c,u,z):={}&\int_\Omega\bl\frac 1p|\nabla z|^p+\frac 12|\nabla  c|^2+W(c,\e(u),z)+f(z)+ \Psi(c)-l(k\tau)\cdot u\br\dx\\
				      & +\frac\delta2\langle A^{k-1}u,u\rangle_{H^2}+\frac\tau2\left\|\frac{z-z_\tau^{k-1}}{\tau}\right\|_{L^2}^2
					+\frac{\tau^2}{2}\left\|\frac{u-2u_\tau^{k-1}+u_\tau^{k-2}}{\tau^2}\right\|_{L^2}^2\\
				      & +\frac{1}{2\tau}\left\|\frac{c-c_\tau^{k-1}}{\tau}\right\|_{V_0}^2
					+\frac{\delta}{2\tau}\left\|\frac{c-c_\tau^{k-1}}{\tau}\right\|_{L^2}^2,
	\end{align*}
	where $V_0= \{ \zeta \in (H^1(\Omega))^* | \langle \zeta, {\bf 1} \rangle_{(H^1)^* \times H^1}=0 \}$. Note that the inverse
	operator
	$A^{k-1,-1}:V_0 \to U_0:=\{ \zeta \in (H^1(\Omega))|  \int_\Omega \zeta \dx =0 \}$ of the operator $A^{k-1}: U_0 \to V_0$ given by
	$$
		u \mapsto \langle \nabla u,
		m(c_\tau^{k-1},z_\tau^{k-1})\nabla \cdot \,  \rangle_{L^2}
	$$
	is well defined.  The space $V_0$ is endowed with the
	scalar product
	$$
		\langle u, v \rangle_{V_0}:= \langle \nabla (A^{-1} u), m(c_\tau^{k-1},z_\tau^{k-1})\nabla (A^{-1}v) \rangle_{L^2}.
	$$
	We refer to \cite{Gar00} for details.

	A minimizer of $\C F_\tau^k$ in the subspace
	\begin{multline}
		\bigg\{c\in H^1(\Omega) \;|\; \int_\Omega (c-c^0) \dx=0\dx\bigg\}\times \Big\{u\in \Hb\;|\;u|_{\Gamma_\mathrm{D}}=b(\tau k)|_{\Gamma_\mathrm{D}}\Big\}\\
		\times \Big\{z\in \Wa\;|\;0\leq z\leq z_\tau^{k-1}\Big\}
	\end{multline}
	obtained by the direct method in the calculus of variations is denoted
        by $(c^k_\tau, u_\tau^k,z_\tau^k)$. More precisely, by a recursive minimization procedure starting from the initial values $(c^0,u^0,z^0)$ and $u^{-1}:=u^0-\tau v^0$,
	we obtain
	functions $(c_\tau^k, u_\tau^k,z_\tau^k)$ for $k=0,\ldots,M_\tau$.
	The velocity field $v_\tau^k$ is set to $(u_\tau^k-u_\tau^{k-1})/\tau$
	and $b_\tau^k$ and $l_\tau^k$ are given by $b(\tau k)$ and $l(\tau k)$.
	
	Let $w_\tau^k\in\{l_\tau^k,b_\tau^k,c^k_\tau,
	u_\tau^k,v_\tau^k,z_\tau^k, \mu_\tau^k \}$, we introduce the piecewise constant
	interpolations $w_\tau$, $w_\tau^-$
	and the linear interpolation $\widehat w_\tau$ with respect to time as
	\begin{align*}
		w_\tau(t)&:=w_\tau^k&&\text{ with }k=\left\lceil t/\tau\right\rceil,\\
		w_\tau^-(t)&:=w_\tau^{\max\{0,k-1\}}&&\text{ with }k=\left\lceil t/\tau\right\rceil,\\
		\widehat w_\tau(t)&:=\beta w_\tau^k+(1-\beta)w_\tau^{\max\{0,k-1\}}&&\text{ with }
			k=\left\lceil t/\tau\right\rceil,\;\beta=\frac{t-(k-1)\tau}{\tau}
	\end{align*}
	and the piecewise constant functions $t_\tau$ and $t_\tau^-$ as
	\begin{align*}
		t_\tau&:=\left\lceil t/\tau\right\rceil\tau=\min\{k\tau\,|\,k\in\mathbb N_0\text{ and }k\tau\geq t\},\\
		t_\tau^-&:=\max\{0,t_\tau-\tau\}.
	\end{align*}
	We would like to remark that, by definition, $w_\tau(t)=w_\tau(t_\tau)$ for all $t\in[0,T]$ and
	\begin{align*}
		\partial_t \widehat v_\tau(t)=\frac{u_\tau^k-2u_\tau^{k-1}+u_\tau^{k-2}}{\tau^2}
	\end{align*}
	for $t\in \left\lceil t/\tau\right\rceil$.
	
	Since the functions $(c^k_\tau, u_\tau^k,z_\tau^k)$ are minimizers, we
	obtain the following necessary conditions (Euler-Lagrange equations)
	by direct methods in the calculus of variations, cf.~\cite{WIAS1520,
	WIASxxxx, HK13_1}:
	\begin{lemma}
	\label{lemma:time_discrete}
		There exists a time-discrete weak solution in the following sense:
		\begin{itemize}
			\item
				spaces:
				\begin{align*}
					&c_\tau,c_\tau^-\in L^\infty(0,T;H^1(\Omega)),&&\widehat c_\tau\in W^{1,\infty}(0,T;H^1(\Omega)),\\
					&u_\tau,v_\tau\in L^\infty(0,T;\Hb),&&\widehat u_\tau,\widehat v_\tau\in W^{1,\infty}(0,T;H^2(\Omega;\R^n)),\\
					&z_\tau,z_\tau^-\in L^\infty(0,T;\Wa),&&\widehat z_\tau\in W^{1,\infty}(0,T;\Wa),\\
					&\mu_\tau\in L^\infty(0,T;H^1(\Omega)),
				\end{align*}
				with
				\begin{align*}
					&c_\tau(0)=c^0\text{ a.e. in }\Omega,\;u_\tau(0)=u^0\text{ a.e. in }\Omega,\;z_\tau(0)=z^0\text{ in }\Omega,\;v_\tau(0)=v^0\text{ a.e. in }\Omega,\\
					&u_\tau=b_\tau\text{ on }{\Gamma_\mathrm{D}}\times(0,T),\;z_\tau\geq 0\text{ a.e. in }\Omega_T,\;\partial_t \widehat z_\tau\leq 0\text{ a.e. in }\Omega_T,
				\end{align*}
			\item
				for all $\zeta\in L^2(0,T;H^1(\Omega))$:
				\begin{align}
				\label{eqn:discrPde1}
					\int_{\Omega_T}(\partial_t \widehat c_\tau)\zeta\dxt=-\int_{\Omega_T} m(c_\tau^-,z_\tau^-)\nabla\mu_\tau\cdot\nabla\zeta\dxt,
				\end{align}
			\item
				for all $\zeta\in H^1(\Omega)$ and for a.e. $t\in(0,T)$:
				\begin{align}
				\label{eqn:discrPde2}
			       \int_{\Omega}\mu_\tau\zeta\dx =	\int_{\Omega}\big(\nabla c_\tau\cdot\nabla \zeta+W_{,c}(c_\tau,\e(u_\tau),z_\tau)\zeta+\Psi'(c_\tau)\zeta
						+\delta(\partial_t \widehat c_\tau)\zeta\big)\dx ,
				\end{align}
			\item
				for all $\zeta\in H_{\Gamma_\mathrm{D}}^2(\Omega;\R^n)$ and for a.e. $t\in(0,T)$:
				\begin{align}
				\label{eqn:discrPde3}
					\int_\Omega\partial_t\widehat v_\tau\cdot\zeta\dx+\int_\Omega W_{,e}(c_\tau,\e(u_\tau),z_\tau):\e(\zeta)\dx+\delta \langle Au_\tau,\zeta\rangle_{H^2}
						=\int_\Omega l_\tau\cdot\zeta\dx,
				\end{align}
			\item
				for a.e. $t\in(0,T)$ and for all $\zeta\in \Wa$ with $0\leq\zeta+z_\tau(t)\leq z_\tau^-(t)$:
				\begin{align}
				\label{eqn:discrPde4}
					0\leq\int_\Omega\bl|\nabla z_\tau|^{p-2}\nabla z_\tau\cdot\nabla\zeta+(W_{,z}(c_\tau,\e(u_\tau),z_\tau)+f'(z_\tau)+\partial_t\widehat z_\tau)\zeta\br\dx.
				\end{align}
		\end{itemize}
	\end{lemma}

	\begin{lemma}[\textbf{A priori estimates}]
	\label{lemma:apriori_tau}
	There exists a constant $C>0$ independent of $\delta$ such that
	\begin{itemize}
	  \item[(i)]
		$\|\nabla c_\tau\|_{L^\infty(0,T;L^2(\Omega;\R^n))} <C, \quad \|\partial_t \widehat c_\tau\|_{L^2(0,T;L^2(\Omega))}<C,$
	  \item[(ii)]
                $\| u_\tau\|_{L^\infty(0,T;H^2(\Omega;\R^n))}<C$, \quad
		$\|v_\tau\|_{L^\infty(0,T;L^2(\Omega;\R^n))}<C,$ \\[2mm]
                $	\|\widehat u_\tau\|_{L^\infty(0,T;H^2(\Omega;\R^n))\cap W^{1,\infty}(0,T;L^2(\Omega;\R^n))}<C$,\\[2mm]
		$		\|\widehat v_\tau\|_{L^{\infty}(0,T;L^2(\Omega;\R^n))\cap H^1(0,T;(H_{\Gamma_\mathrm{D}}^2(\Omega;\R^n))^*)}<C$,
	  \item[(iii)]
		$\|\nabla z_\tau\|_{L^\infty(0,T;L^p(\Omega;\R^n))}<C,\quad \|\partial_t \widehat z_\tau\|_{L^2(0,T;L^2(\Omega))}<C,$
	  \item[(iv)]
		$\|\nabla\mu_\tau\|_{L^2(0,T;L^2(\Omega;\R^n))}<C, \quad \|
		  m(c^-_\tau, z^-_\tau)^{1/2} \nabla\mu_\tau\|_{L^2(0,T;L^2(\Omega;\R^n))}<C.$
	\end{itemize}
	\end{lemma}

\begin{proof} We split the proof into two steps.
        We first prove the a priori estimates (i), (ii) and (iv) and then we deduce estimate (iii). \\
	\textit{First a priori estimates.}
	Testing \eqref{eqn:discrPde1} with $\tau\mu_\tau$,
	testing\eqref{eqn:discrPde2} with $c_\tau-c_\tau^-$,
	testing \eqref{eqn:discrPde3} with $u_\tau-u_\tau^--(b_\tau-b_\tau^-)$,
	and adding everything, yield
	\begin{align*}
		T_1(t)+T_2(t)+T_3(t)+T_4(t)+T_5(t)\leq 0
	\end{align*}
	with
	{\small\begin{align*}
		&T_1(t):=\int_\Omega \nabla c_\tau(t)\cdot\nabla(c_\tau(t)-c_\tau^-(t))\dx
			+\int_\Omega \delta\langle\nabla(\nabla u_\tau(t)),\nabla(\nabla(u_\tau(t)-u_\tau^-(t)))\rangle\dx\\
			&\qquad\qquad+\int_\Omega \partial_t \widehat v_\tau(t)\cdot(u_\tau(t)-u_\tau^-(t))\dx,\\
		&T_2(t):=\tau \int_\Omega m(c_\tau(t),z_\tau(t))|\nabla\mu_\tau(t)|^2\dx+\tau\int_\Omega \delta|\partial_t \widehat c_\tau(t)|^2\dx,\\
		&T_3(t):=\int_\Omega W_{,c}(c_\tau(t),\e(u_\tau(t)),z_\tau(t))(c_\tau(t)-c_\tau^-(t))\dx\\
			&\qquad\qquad+\int_\Omega W_{,e}(c_\tau(t),\e(u_\tau(t)),z_\tau(t)):\e(u_\tau(t)-u_\tau^-(t))\dx\\
		&T_4(t):=\tau\int_\Omega \Psi'(c_\tau(t))\partial_t\widehat c_\tau(t)\dx
			-\tau\int_\Omega l_\tau(t)\cdot\partial_t\widehat u_\tau(t)\dx\\
		&T_5(t):=-\tau\int_\Omega\Big(\partial_t\widehat v_\tau(t)\cdot\partial_t \widehat b_\tau(t)
			+W_{,e}(c_\tau(t),\e(u_\tau(t)),z_\tau(t)):\e(\partial_t \widehat b_\tau(t))\Big)\dx,\\
			&\qquad\qquad-\tau\int_\Omega\Big(\delta\langle\nabla(\nabla u_\tau(t)),\nabla(\nabla(\partial_t \widehat b_\tau(t)))\rangle_{\R^{n\times n\times n}}
			-l_\tau(t)\cdot\partial_t \widehat b_\tau(t)\Big)\dx.
	\end{align*}}These terms are estimated in the following.
	\begin{itemize}
		\item
			Convexity estimates yield
			\begin{align*}
				T_1(t)\geq{}& \frac12\|\nabla c_\tau(t)\|_{L^2(\Omega)}^2-\frac12\|\nabla c_\tau^-(t)\|_{L^2(\Omega)}^2\\
					&+\frac\delta2\|\nabla (\nabla u_\tau(t))\|_{L^2(\Omega;\R^{n\times n\times n})}^2
					-\frac\delta2\|\nabla (\nabla u_\tau^-(t))\|_{L^2(\Omega;\R^{n\times n\times n})}^2\\
					&+\frac12\|v_\tau(t)\|_{L^2(\Omega;\R^n)}^2-\frac12\|v_\tau^-(t)\|_{L^2(\Omega;\R^n)}^2.
			\end{align*}
		\item
			We obtain for small $\eta>0$:
			\begin{align*}
				T_2(t)\geq{}& \eta\int_{t_\tau^-}^{t_\tau}\Big(\|\nabla\mu_\tau(s)\|_{L^2(\Omega;\R^n)}^2+ \delta \|\partial_t\widehat
	c_\tau(s)\|_{L^2(\Omega)}^2\Big)\ds.
			\end{align*}
		\item
			By the convexity argument and by $z_\tau\leq z_\tau^-$, we gain
			\begin{align*}
				&W_{,e}(c_\tau(t),\e(u_\tau(t)),z_\tau(t)):\e(u_\tau(t)-u_\tau^-(t))\\
				&\qquad\geq W(c_\tau(t),\e(u_\tau(t)),z_\tau(t))-W(c_\tau(t),\e(u_\tau^-(t)),z_\tau(t))\\
				&\qquad\geq W(c_\tau(t),\e(u_\tau(t)),z_\tau(t))-W(c_\tau^-(t),\e(u_\tau^-(t)),z_\tau^-(t))\\
				&\qquad\quad+\int_{t_\tau^-}^{t_\tau}W_{,c}(\widehat c_\tau(s),\e(u_\tau^-(s)),z_\tau(s))\partial_t\widehat c_\tau(s)\ds,
			\end{align*}
			and conclude ($\eta>0$ is chosen as small as necessary)
			\begin{align*}
				&T_3(t)\geq\int_\Omega \big(W(c_\tau(t),\e(u_\tau(t)),z_\tau(t))-W(c_\tau^-(t),\e(u_\tau^-(t)),z_\tau^-(t))\big)\dx\\
					&\qquad\quad+\int_{t_\tau^-}^{t_\tau}\int_\Omega W_{,c}(c_\tau(s),\e(u_\tau(s)),z_\tau(s))\partial_t \widehat c_\tau(s)\dxs\\
					&\qquad\qquad+\int_{t_\tau^-}^{t_\tau}\int_\Omega W_{,c}(\widehat c_\tau(s),\e(u_\tau^-(s)),z_\tau(s))\partial_t\widehat c_\tau(s)\dxs\\
				&\qquad\geq
				\int_\Omega \big(W(c_\tau(t),\e(u_\tau(t)),z_\tau(t))-W(c_\tau^-(t),\e(u_\tau^-(t)),z_\tau^-(t))\big)\dx\\
					&\qquad\quad-C_\eta\int_{t_\tau^-}^{t_\tau}\|W_{,c}(c_\tau(s),\e(u_\tau(s)),z_\tau(s))\|_{L^2(\Omega)}^2\ds\\
					&\qquad\quad-C_\eta\int_{t_\tau^-}^{t_\tau}\|W_{,c}(\widehat c_\tau(s),\e(u_\tau^-(s)),z_\tau(s))\|_{L^2(\Omega)}^2\ds\\
					&\qquad\quad-\eta\int_{t_\tau^-}^{t_\tau}\|\partial_t \widehat c_\tau(s)\|_{L^2(\Omega)}^2\ds\\
				&\qquad\geq\int_\Omega \big(W(c_\tau(t),\e(u_\tau(t)),z_\tau(t))-W(c_\tau^-(t),\e(u_\tau^-(t)),z_\tau^-(t))\big)\dx\\
					&\qquad\quad-\widehat C_\eta\int_{t_\tau^-}^{t_\tau}\big(\|c_\tau(s)\|_{L^2(\Omega)}^2+\|c_\tau^-(s)\|_{L^2(\Omega)}^2+\|\e(u_\tau(s))\|_{L^2(\Omega)}^2\big)\ds\\
					&\qquad\quad-\widehat C_\eta\int_{t_\tau^-}^{t_\tau}\|\e(u_\tau^-(s))\|_{L^2(\Omega)}^2\ds
						-\eta\int_{t_\tau^-}^{t_\tau}\|\partial_t \widehat c_\tau(s)\|_{L^2(\Omega)}^2\ds.
			\end{align*}
		\item
			Convexity of $\Psi_1$ combined with growth condition \eqref{eqn:assumption_psi2} and Young's inequality show
			\begin{align*}
				T_4(t)\geq{}&\int_\Omega\Psi_1(c_\tau(t))\dx-\int_\Omega\Psi_1(c_\tau^-(t))\dx
\end{align*}\begin{align*}
					&-\eta\int_{t_\tau^-}^{t_\tau}\Big(\|\partial_t\widehat c_\tau(s)\|_{L^2(\Omega)}^2+\|l_\tau(s)\|_{L^2(\Omega;\R^n)}^2\Big)\ds\\
				&-C_\eta\int_{t_\tau^-}^{t_\tau}\Big(\|\Psi_{2}'(c_\tau(s))\|_{L^2(\Omega)}^2
					+\|v_\tau(s)\|_{L^2(\Omega;\R^n)}^2 \Big)\ds\\
				\geq{}&\int_\Omega\Psi_1(c_\tau(t))\dx-\int_\Omega\Psi_1(c_\tau^-(t))\dx\\
					&-\eta\int_{t_\tau^-}^{t_\tau}\Big(\|\partial_t\widehat c_\tau(s)\|_{L^2(\Omega)}^2+\|l_\tau(s)\|_{L^2(\Omega;\R^n)}^2\Big)\ds\\
				&-C_\eta\int_{t_\tau^-}^{t_\tau}\Big(\|c_\tau(s)\|_{L^2(\Omega)}^2+\|v_\tau(s)\|_{L^2(\Omega;\R^n)}^2
                                        \Big)\ds.
			\end{align*}
		\item
			By using the discrete integration by parts formula, i.e.,
			\begin{align}
				\int_{t_\tau^-}^{t_\tau}\int_{\Omega}\partial_t \widehat v_\tau\cdot\partial_t \widehat b_\tau\dxs
					={}&\int_\Omega v_\tau(t)\cdot\partial_t \widehat b_\tau(t)\dx
					-\int_\Omega v_\tau^-(t)\cdot\partial_t\widehat b_\tau(t-\tau)\dx\notag\\
				&-\int_{t_\tau^-}^{t_\tau}\int_\Omega v_\tau^-(s)\cdot\frac{\partial_t\widehat b_\tau(s)-\partial_t\widehat b_\tau(s-\tau)}{\tau}\dxs,
			\label{eqn:discrIntegrByParts}
			\end{align}
			we obtain
			\begin{align*}
				T_5(t)\geq{}&-\int_\Omega v_\tau(t)\cdot\partial_t \widehat b_\tau(t)\dx
					+\int_\Omega v_\tau^-(t)\cdot\partial_t\widehat b_\tau(t-\tau)\dx\\
					&-\int_{t_\tau^-}^{t_\tau}\Big(\eta\|v_\tau^-(s)\|_{L^2(\Omega)}^2
					+C_\eta\Big\|\frac{\partial_t\widehat b_\tau(s)-\partial_t\widehat b_\tau(s-\tau)}{\tau}\Big\|_{L^2(\Omega)}^2\Big)\ds\\
					&-\int_{t_\tau^-}^{t_\tau}\Big(\eta\|c_\tau(s)\|_{L^2(\Omega)}^2+\eta\|\e(u_\tau(s))\|_{L^2(\Omega;\R^{n\times n})}^2\Big)\ds\\
					&-\int_{t_\tau^-}^{t_\tau}C_\eta\|\e(\partial_t \widehat b_\tau(s))\|_{L^2(\Omega;\R^{n\times n})}^2\ds\\
					&-\int_{t_\tau^-}^{t_\tau}\Big(\eta\|\nabla(\nabla u_\tau(s))\|_{L^2(\Omega;\R^{n\times n\times n})}^2
					+C_\eta\|\nabla(\nabla \partial_t \widehat b_\tau(s))\|_{L^2(\Omega;\R^{n\times n\times n})}^2\Big)\ds\\
					&-\int_{t_\tau^-}^{t_\tau}\Big(\eta\|l_\tau(s)\|_{L^2(\Omega;\R^n)}^2+C_\eta\|\partial_t\widehat
	b_\tau(s)\|_{L^2(\Omega;\R^n)}^2\Big)\ds .
			\end{align*}
	\end{itemize}
	Summing over the discrete time points $t_\tau=0,\tau,\ldots, k\tau$ for an arbitrary but fixed chosen $k\in\N$, we can apply Gronwall's inequality
	and obtain the following boundedness properties:
	\begin{align}
		\|\nabla c_\tau\|_{L^\infty(0,T;L^2(\Omega;\R^n))}&<C,  \\
		\|\partial_t \widehat c_\tau\|_{L^2(0,T;L^2(\Omega))}&<C, \\
		\|\nabla(\nabla u_\tau)\|_{L^\infty(0,T;L^2(\Omega;\R^{n\times
	n\times n}))}&<C, \label{eqn:aux1}\\
		\|\e(u_\tau)\|_{L^\infty(0,T;L^2(\Omega;\R^{n\times n}))}&<C, \label{eqn:aux2}\\
		\|v_\tau\|_{L^\infty(0,T;L^2(\Omega;\R^n))}&<C, \label{eqn:aux3}\\
		\|\nabla\mu_\tau\|_{L^2(0,T;L^2(\Omega;\R^n))}&<C,
	\end{align}
	where $C>0$ is independent of $\tau$. Combining estimates \eqref{eqn:aux1}-\eqref{eqn:aux3}  with
	Korn's inequality, we obtain
	\begin{align*}
				\|u_\tau\|_{L^\infty(0,T;H^2(\Omega;\R^n))}&<C.
			\end{align*}
			Consequently, by noticing $v_\tau=\partial_t\widehat u_\tau$,
			\begin{align*}
				\|\widehat u_\tau\|_{L^\infty(0,T;H^2(\Omega;\R^n))\cap W^{1,\infty}(0,T;L^2(\Omega;\R^n))}&<C.
			\end{align*}
			A comparison argument in \eqref{eqn:discrPde3} also gives
			\begin{align*}
				\|\widehat v_\tau\|_{L^{\infty}(0,T;L^2(\Omega;\R^n))\cap H^1(0,T;(H_{\Gamma_\mathrm{D}}^2(\Omega;\R^n))^*)}&<C.
			\end{align*}
	
	\textit{Second a priori estimates.}
	Testing \eqref{eqn:discrPde4} with $z_\tau^--z_\tau$, yields
	\begin{align*}
		&\int_\Omega |\nabla z_\tau(t)|^{p-2}\nabla z_\tau(t)\cdot\nabla (z_\tau(t)-z_\tau^-(t))\dx
			+\frac 12\tau\|\partial_t\widehat z_\tau(t)\|_{L^2(\Omega)}^2\\
		&\qquad\qquad\leq - \tau\int_\Omega\bl W_{,z}(c(\tau(t)),\e(u_\tau(t)),z_\tau(t))\partial_t\widehat z_\tau(t)+f'(z_\tau(t))\partial_t\widehat z_\tau(t)\br\dx.
	\end{align*}
	Now we apply a convexity estimate and get
	\begin{align*}
		&\frac 1p\|\nabla z_\tau(t)\|_{L^p(\Omega;\R^n)}^p-\frac 1p\|\nabla z_\tau^-(t)\|_{L^p(\Omega;\R^n)}^p
			+\frac 12\tau\|\partial_t\widehat z_\tau(t)\|_{L^2(\Omega)}^2\\
		&\qquad\qquad\leq\tau\eta\|\partial_t\widehat z_\tau(t)\|_{L^2(\Omega)}^2
			+\tau C_\eta\big(1+\|c_\tau(t)\|_{L^4(\Omega)}^4+\|\e(u_\tau(t))\|_{L^4(\Omega;\R^{n\times n})}^4\big).
	\end{align*}
	We end up with
	\begin{align*}
		\|\nabla z_\tau\|_{L^\infty(0,T;L^p(\Omega;\R^n))}&<C,\\
		\|\partial_t \widehat z_\tau\|_{L^2(0,T;L^2(\Omega))}&<C,
	\end{align*}
	where $C>0$ is independent of $\tau$.
\end{proof}

	By applying Poincar{\'e}'s inequality, standard weak and weakly-star compactness results to the above a priori estimates,
	we obtain the following convergence properties.
	\begin{lemma}[Convergence properties]
		There exist functions
		\begin{align*}
			&c\in L^\infty(0,T;H^1(\Omega))\cap H^1(0,T;L^2(\Omega)),\\
			&u\in L^\infty(0,T;H^2(\Omega;\R^n))\cap W^{1,\infty}(0,T;L^2(\Omega;\R^n))\cap W^{2,\infty}(0,T;(H_{\Gamma_\mathrm{D}}^2(\Omega;\R^n))^*),\\
			&z\in L^\infty(0,T;\Wa)\cap H^1(0,T;L^2(\Omega)),\\
			&\mu\in L^2(0,T;H^1(\Omega))
		\end{align*}
		and subsequences (omitting the subscript) such that for all $r\geq 1$ and $s<2^*$:
		\begin{align}
		  \label{eqn:conv0}
			c_{\tau},c_{\tau}^-&\to c&&\text{ weakly-star in }L^\infty(0,T;H^1(\Omega)),\\
		  \label{eqn:conv1}
			&&&\text{ strongly in }L^r(0,T;L^s(\Omega)),\text{ a.e.~in }\Omega_T,\\
		  \label{eqn:conv2}
			\widehat c_{\tau}&\to c&&\text{ weakly-star in }
\\ && & \hspace{0.3cm}L^\infty(0,T;H^1(\Omega))\cap H^1(0,T;L^2(\Omega)),\\
		  \label{eqn:conv3_0}
			u_{\tau},u_{\tau}^-&\to u&&\text{ weakly-star in }
L^\infty(0,T;H^2(\Omega;\R^n)),\\
		  \label{eqn:conv3}
			&&&\text{ strongly in }L^r(0,T;H^1(\Omega;\R^n)),\text{ a.e.~in }\Omega_T,\\
		  \label{eqn:conv4}
			\widehat u_{\tau}&\to u&&\text{ weakly-star in }\notag
\\&& & \hspace{0.3cm}
L^\infty(0,T;H^2(\Omega;\R^n))\cap W^{1,\infty}(0,T;L^2(\Omega;\R^n)),\\
		  \label{eqn:conv5}
			v_{\tau},v_{\tau}^-&\to \partial_t u&&\text{ weakly-star in }L^{\infty}(0,T;L^2(\Omega;\R^n)),
\end{align}\begin{align}		  \label{eqn:conv6}
			\widehat v_{\tau}&\to \partial_t u&&\text{ weakly-star
                          in }
\\&& & \hspace{0.3cm}
L^{\infty}(0,T;L^2(\Omega;\R^n))\cap H^1(0,T;(H_{\Gamma_\mathrm{D}}^2(\Omega;\R^n))^*),\\
			z_{\tau},z_{\tau}^-&\to z&&\text{ weakly-star in }L^\infty(0,T;\Wa),\notag\\
		  \label{eqn:conv7}
			&&&\text{ strongly in }L^r(0,T;L^r(\Omega;\R^n)),\text{ a.e.~in }\Omega_T,\\
		  \label{eqn:conv8}
			\widehat z_{\tau}&\to z&&\text{ weakly-star in }
\\&& & \hspace{0.3cm}
L^\infty(0,T;\Wa)\cap H^1(0,T;L^2(\Omega)),\\
		  \label{eqn:conv9}
			\mu_{\tau}&\to \mu&&\text{ weakly in } L^2(0,T;H^1(\Omega)),\\
		  \label{eqn:conv10}
			m(c_\tau^-, z_\tau^-)^{\frac{1}{2}} \nabla \mu_{\tau} &\to  m(c, z)^{\frac{1}{2}} \nabla \mu&&\text{ weakly in } L^2(0,T;L^2(\Omega; \R^n))
		\end{align}
		as $\tau\searrow 0$.
	\end{lemma}
		
	Strong convergence of a subsequence of $\{\nabla z_{\tau}\}$ in $L^p(\Omega_T;\R^n)$ can
	be shown as in \cite{HK13_1} by a tricky approximation argument.
	\begin{lemma}[cf.~\cite{HK13_1}]
	\label{lemma:strongConvZ}
		There exists a sequence $\{ \tau_k \}_{ k \in \N}$ such that
		$z_{\tau_k}\to z$ in $L^p(0,T;W^{1,p}(\Omega))$ as $\tau_k \searrow 0$.
	\end{lemma}
  For a time discrete solution of the regularized system,
	we can prove the validity of an energy inequality of type \eqref{eqn:weak6}
	except the additional discretization error terms $e_\tau^1,\ldots,e_\tau^4$
	which will turn out to converge to $0$ in a certain sense as $\tau\searrow 0$.
	\begin{lemma} {\bf (Discrete energy inequality)}
	\label{lemma:energy_time_discrete}
		Let a time-discrete weak solution be given as in
		Lemma \ref{lemma:time_discrete}.  Then the following energy estimate is satisfied for a.e. $t\in(0,T)$:
		\begin{align}
			&\C E_\tau(t)+\C K_\tau(t)+\C D_\tau(0,t)+\int_0^{t_\tau}\bl e_\tau^1(s)+e_\tau^2(s)+e_\tau^3(s)+e_\tau^4(s)\br\ds\notag\\
		\label{eqn:discrEI}
			&\qquad\leq \C E_\tau(0)+\C K_\tau(0)+\C W_{\tau,\mathrm{ext}}(0,t)
		\end{align}
		with the discrete energies
		\begin{align*}
			&\C E_\tau(t):=\int_\Omega\bl\frac 1p|\nabla z_\tau(t)|^p+\frac 12|\nabla c_\tau(t)|^2+W(c_\tau(t),\e(u_\tau(t)),z_\tau(t))
				+f(z_\tau(t))\br\dx\\
			&\hspace*{4em} + \int_\Omega\Psi(c_\tau(t))\dx+\frac\delta2 \langle A u_\tau(t),u_\tau(t)\rangle_{H^2},\\
			&\C K_\tau(t):=\int_\Omega\frac12|v_\tau(t)|^2\dx,\\
			&\C D_\tau(0,t):=\int_0^{t_\tau}\int_\Omega\big(|\partial_t\widehat z_\tau|^2+\delta|\partial_t\widehat c_\tau|^2 +
				m(c_\tau^-,z_\tau^-)\nabla\mu_\tau\cdot\nabla\mu_\tau\big)\dxs,\\
			&\C W_{\tau,\mathrm{ext}}(0,t):=\int_0^{t_\tau}\int_\Omega W_{,e}(c_\tau,\e(u_\tau),z_\tau):\e(\partial_t\widehat b_\tau)\dxs\\
			&\qquad\qquad\qquad+\delta\int_0^{t_\tau}\langle A u_\tau(s),\partial_t\widehat b_\tau(s)\rangle_{H^2}\ds\\
			&\qquad\qquad\qquad +\int_0^{t_\tau}\int_\Omega l_\tau\cdot\bl \partial_t\widehat u_\tau-\partial_t\widehat b_\tau\br\dxs
				-\int_\Omega v^0\cdot\partial_t\widehat b_\tau(0)\dx\\
	&\qquad\qquad\qquad +\int_\Omega v_\tau(t)\cdot\partial_t \widehat b_\tau(t)\dx-\int_0^{t_\tau}\int_\Omega v_\tau^-(s)\cdot\frac{\partial_t\widehat b_\tau(s)-\partial_t\widehat b_\tau(s-\tau)}{\tau}\dxs,
\end{align*}
		and the error terms
		\begin{align*}
			e_\tau^1(t):={}&\int_\Omega\bl \frac{ W(c_\tau(t),\e(u_\tau^-(t)),z_\tau^-(t))-W(c_\tau(t),\e(u_\tau^-(t)),z_\tau(t))}{\tau}\br\dx \\
				& \qquad +\int_\Omega W_{,z}(c_\tau(t),\e(u_\tau(t)),z_\tau(t))\,\partial_t \widehat z_\tau(t)\dx,\\
			e_\tau^2(t):={}&\int_\Omega \bl \frac{W(c_\tau^-(t),\e(u_\tau^-(t)),z_\tau^-(t))-W(c_\tau(t),\e(u_\tau^-(t)),z_\tau^-(t))}{\tau}\br\dx\\
				& \qquad+\int_\Omega W_{,c}(c_\tau(t),\e(u_\tau(t)),z_\tau(t))\,\partial_t \widehat c_\tau(t)\dx,\\
			e_\tau^3(t):={}&\int_\Omega\frac{\Psi(c^-_\tau(t))-\Psi(c_\tau(t))}{\tau}\dx+
											\int_\Omega \Psi'(c_\tau(t))\,\partial_t \widehat c_\tau(t)\dx,\\
			e_\tau^4(t):={}&\int_\Omega\frac{f(z^-_\tau(t))-f(z_\tau(t))}{\tau}\dx+											\int_\Omega f'(z_\tau(t))\,\partial_t \widehat z_\tau(t)\dx.
		\end{align*}
	\end{lemma}
	\begin{proof}
	We compute by using convexity of $W$ with respect to $e$:
	\begin{align}
	  \int_\Omega W_{,e}(c_\tau, \e(u_\tau),z_\tau)
		& :\e(u_\tau-u_\tau^-)\dx \notag\\
		& \geq\int_\Omega \bl W(c_\tau,\e(u_\tau),z_\tau)-W(c_\tau^-,\e(u_\tau^-),z_\tau^-)\br\dx  \notag\\
		&\qquad\quad+\int_\Omega\bl W(c_\tau,\e(u_\tau^-),z_\tau^-)-W(c_\tau,\e(u_\tau^-),z_\tau)\br\dx \notag \\
		&\qquad\quad+\int_\Omega \bl	W(c_\tau^-,\e(u_\tau^-),z_\tau^-)-W(c_\tau,\e(u_\tau^-),z_\tau^-)\br\dx.    \label{eqn:W_c_est1}
	 \end{align}
	We test \eqref{eqn:discrPde3} with $u_\tau-u_\tau^--(b_\tau-b_\tau^-)$, apply \eqref{eqn:W_c_est1}, use further convexity arguments
	and end up with
	\begin{align}
		&\frac12\left\|v_\tau(t)\right\|_{L^2}^2-\frac12\left\|v_\tau^-(t)\right\|_{L^2}^2
			+\frac\delta2 \langle A u_\tau(t),u_\tau(t)\rangle_{H^2}-\frac\delta2 \langle A u_\tau^-(t),u_\tau^-(t)\rangle_{H^2}\notag\\
		&+\int_\Omega\bl W(c_\tau(t),\e(u_\tau(t)),z_\tau(t))-W(c^-_\tau(t),\e(u_\tau^-(t)),z_\tau^-(t))\br\dx\\
		&-\int_\Omega\partial_t \widehat v_\tau(t)\cdot\bl b_\tau(t)-b_\tau^-(t)\br\dx\notag\\
		&+ \int_\Omega\bl W(c_\tau(t),\e(u_\tau^-(t)),z_\tau^-(t))-W(c_\tau(t),\e(u_\tau^-(t)),z_\tau(t))\br\dx \notag \\
                &  +\int_\Omega \bl W(c_\tau^-(t),\e(u_\tau^-(t)),z_\tau^-(t))-W(c_\tau(t),\e(u_\tau^-(t)),z_\tau^-(t))\br\dx  \notag\\
		&\qquad\leq \int_\Omega l_\tau(t)\cdot\bl u_\tau(t)-u_\tau^-(t)-(b_\tau(t)-b_\tau^-(t))\br\dx\notag\\
		&\qquad\quad+\int_\Omega W_{,e}(c_\tau(t),\e(u_\tau(t)),z_\tau(t)):\e(b_\tau(t)-b_\tau^-(t))\dx\notag\\
		&\qquad\quad+\delta\langle A u_\tau(t),b_\tau(t)-b_\tau^-(t)\rangle_{H^2}.
	\label{eqn:testEq21}
	\end{align}
	Using the convexity estimate
	$$\int_\Omega |\nabla z_\tau|^{p-2}\nabla z_\tau\cdot\nabla(z_\tau-z_\tau^-)\dx
		\geq\frac 1p\|\nabla z_\tau\|_{L^p}^p-\frac1p\|\nabla z_\tau^-\|_{L^p}^p $$
	and testing \eqref{eqn:discrPde4} with $z_\tau^--z_\tau$ yield
	\begin{align}
		&\frac 1p\|\nabla z_\tau(t)\|_{L^p}^p-\frac1p\|\nabla z_\tau^-(t)\|_{L^p}^p +\tau\left\|\partial_t\widehat z_\tau(t)\right\|_{L^2}^2\notag\\
	\label{eqn:testEq2}
		&\qquad\leq \int_\Omega\bl W_{,z}(\e(u_\tau(t)),z_\tau(t))+f'(z_\tau(t))\br(z_\tau^-(t)-z_\tau(t))\dx.
	\end{align}

	Next we test equation \eqref{eqn:discrPde1} with $\tau\mu_\tau$
        and \eqref{eqn:discrPde2} with $ (c_\tau - c_\tau^-)$ and add the two
        derived equations. We obtain
	by means of the convexity property
	$$
		\int_\Omega \nabla c_\tau \cdot \nabla(c_\tau-c_\tau^-)\dx \geq\frac 12 \|\nabla c_\tau\|_{L^2}^2-\frac1 2\|\nabla c_\tau^-\|_{L^2}^2
	$$
	the estimate
	\begin{multline}
	\label{eqn:testEq3}
	    \frac 12\|\nabla c_\tau(t)\|_{L^2}^2-\frac12\|\nabla c_\tau^-(t)\|_{L^2}^2
	    +\int_\Omega \Big( W_{,c}(c_\tau(t),\e(u_\tau(t)),z_\tau(t))(c_\tau(t)-c_\tau^-(t))+ \\
\Psi'(c_\tau(t))(c_\tau(t)-c_\tau^-(t)) +\tau \, m(c_\tau^-(t),z_\tau^-(t))\nabla\mu_\tau(t)\cdot  \nabla\mu_\tau(t) \Big)\dx
	    +\delta \, \tau \| \partial_t \widehat c_\tau(t)\|^2_{L^2}  \le 0 .
	\end{multline}
	Adding the estimates \eqref{eqn:testEq21}--\eqref{eqn:testEq3}, we end up with
	\begin{align*}
		&\frac12\left\|v_\tau(t)\right\|_{L^2}^2-\frac12\left\|v_\tau^-(t)\right\|_{L^2}^2
			+\frac\delta2 \langle A u_\tau(t),u_\tau(t)\rangle_{H^2}-\frac\delta2 \langle A u_\tau^-(t),u_\tau^-(t)\rangle_{H^2}\\
		&\quad +  \frac 12\|\nabla c_\tau(t)\|_{L^2}^2-\frac12\|\nabla c_\tau^-(t)\|_{L^2}^2
		  +\frac 1p\|\nabla z_\tau(t)\|_{L^p}^p-\frac1p\|\nabla z_\tau^-(t)\|_{L^p}^p\\
		&\quad +\tau\Big(
		   \left\|\partial_t\widehat z_\tau(t)\right\|_{L^2}^2
		  +  \delta\left\|\partial_t\widehat  c_\tau(t)\right\|_{L^2}^2    \
		  +\int_\Omega m(c_\tau^-(t),z_\tau^-(t))\nabla\mu_\tau(t)\cdot  \nabla\mu_\tau(t)  \dx \Big)\\
		& \quad -\int_\Omega\partial_t \widehat v_\tau(t)\cdot\bl
                  b_\tau(t)-b_\tau^-(t)\br\dx
 \\ & \qquad \quad
		  +\int_\Omega\bl W(c_\tau(t),\e(u_\tau(t)),z_\tau(t))-W(c^-_\tau(t),\e(u_\tau^-(t)),z_\tau^-(t))\br\dx\\
		&\quad+\int_\Omega f(z_\tau(t))\dx-\int_\Omega f(z_\tau^-(t))\dx+ \int_\Omega \Psi(c_\tau(t))\dx\\
		&\quad-\int_\Omega
		  \Psi(c_\tau^-(t))\dx + \tau \big( e_\tau^1(t)+ e_\tau^2(t) +  e_\tau^3(t)+ e_\tau^4(t)  \big)\\
		&\qquad\leq \int_\Omega l_\tau(t)\cdot\bl u_\tau(t)-u_\tau^-(t)-(b_\tau(t)-b_\tau^-(t))\br\dx\\
		&\qquad\quad
+ \int_\Omega
W_{,e}(c_\tau(t),\e(u_\tau(t)),z_\tau(t)):\e(b_\tau(t)-b_\tau^-(t))\dx
 \\ & \qquad \quad
		  +\delta\langle A u_\tau(t),b_\tau(t)-b_\tau^-(t)\rangle_{H^2}
	\end{align*}
	with the error terms $e_\tau^1(t)$, $e_\tau^2(t)$, $e_\tau^3(t)$ and $e_\tau^4(t)$.
	Summing over the discrete time points and taking into account the discrete integration by parts
	formula \eqref{eqn:discrIntegrByParts}, we finally obtain the claim.
        $\phantom{w}$\hspace{3cm}
		\end{proof}
	\vspace*{4pt}\noindent{\it{Proof of Theorem \ref{theorem:mainResult_delta}.}}
	  We are going to establish the equalities and inequalities of the weak
	  formulation \eqref{eqn:weak1_delta}-\eqref{eqn:weak6_delta}.
	  \begin{itemize}
	    \item (Cahn-Hilliard equation)\\
	      Because of the convergence properties \eqref{eqn:conv2}, \eqref{eqn:conv1}, \eqref{eqn:conv7}
	      and \eqref{eqn:conv9}  we may pass to the limit in \eqref{eqn:discrPde1} and obtain \eqref{eqn:weak1_delta}.

	      To establish \eqref{eqn:weak2_delta}, we first integrate \eqref{eqn:discrPde2}
	      over time from $t=0$ to $t=T$. The growth condition \eqref{eqn:assumption_psi} and
	      the convergence properties \eqref{eqn:conv9}, \eqref{eqn:conv0}, \eqref{eqn:conv1},
	      \eqref{eqn:conv3},  \eqref{eqn:conv7}  and \eqref{eqn:conv2}
	      allow us to pass to the limit in the integrated version of \eqref{eqn:discrPde2} which shows \eqref{eqn:weak2_delta}.
	    \item (Balance equation of forces)\\
	      By using the canonical embedding $L^2(\Omega;\R^n)\hookrightarrow (H_{\Gamma_\mathrm{D}}^2(\Omega;\R^n))^*$, it follows for all
	      $\zeta\in H_{\Gamma_\mathrm{D}}^2(\Omega;\R^n)$
	      $$\int_\Omega\partial_t\widehat v_\tau(t)\cdot\zeta\dx=\langle\partial_t\widehat v_\tau(t),\zeta\rangle_{H^2}.$$
	      Keeping this identity in mind, integrating \eqref{eqn:discrPde3} from $t=0$ to $t=T$ and passing to the limit $\tau\searrow0$ by using
	      \eqref{eqn:conv6} and \eqref{eqn:conv1},   \eqref{eqn:conv3}, \eqref{eqn:conv7} and \eqref{eqn:conv3_0}, we obtain \eqref{eqn:weak3_delta}.
	    \item (Variational inequality for $z$)\\
	      To obtain the variational inequalities \eqref{eqn:weak4_delta} and \eqref{eqn:weak5_delta}, we can proceed as in \cite{HK13_1}.  In particular,
	      \eqref{eqn:weak4_delta} is valid for the subgradient
	      \begin{align}
		\label{eqn:xiDef}
		  \xi=-\chi_{\{z=0\}}\mathrm{max}\Big\{0,W_{,z}(c,\e(u),z)+f'(z)\Big\},	      \end{align}
		which satisfies \eqref{eqn:weak5_delta}, where $\chi_{\{z=0\}}$ is the characteristic function of the set $\{z=0\}$.
	    \item (Energy inequality)\\
	      To treat the energy inequality  \eqref{eqn:discrEI}, we set
	      \begin{align*}
		&A_\tau(t):={}\\
		  &\quad\int_\Omega\Big(\frac 1p|\nabla z_\tau(t)|^p+\frac 12|\nabla c_\tau(t)|^2+W(c(t),\e(u_\tau(t)),z_\tau(t))+f(z_\tau(t))+ \Psi(c_\tau(t))\Big)\dx\\
		  &\quad-\int_\Omega\bl\frac1p|\nabla z^0|^p + \frac 12 |\nabla c^0|^2+W(c^0,\e(u^0),z^0)+f(z^0)+ \Psi(c^0)\br\dx\\
		  &\quad+\int_\Omega\frac12|v_\tau(t)|^2\dx-\int_\Omega\frac12|v^0|^2\dx
		   +\frac\delta2 \langle A u_\tau(t),u_\tau(t)\rangle_{H^2}-\frac\delta2 \langle A u^0,u^0\rangle_{H^2}\\
		  &\quad-\int_\Omega v_\tau(t)\cdot\partial_t \widehat b_\tau(t)\dx
		  	+\int_\Omega v^0\cdot\partial_t\widehat b_\tau(0)\dx\\
		&B_\tau(t):={}\\
		  &\quad\int_0^{t_\tau}\int_\Omega \Big(|\partial_t\widehat z_\tau|^2   + \delta |\partial_t\widehat c_\tau|^2
		    +m(c_\tau,z_\tau)\nabla\mu_\tau\cdot  \nabla\mu_\tau   \Big)\dxs\\
		  &\quad-\int_0^{t_\tau}\int_\Omega l_\tau\cdot\bl \partial_t\widehat u_\tau-\partial_t\widehat b_\tau\br\dxs\\
		  &\quad-\int_0^{t_\tau}\int_\Omega W_{,e}(c_\tau, \e(u_\tau),z_\tau):\e(\partial_t\widehat b_\tau)\dxs
		   -\delta\int_0^{t_\tau}\langle A u_\tau(s),\partial_t\widehat b_\tau(s)\rangle_{H^2}\ds\\
		  &\quad+\int_0^{t_\tau}\int_\Omega v_\tau^-(s)\cdot\frac{\partial_t\widehat b_\tau(s)-\partial_t\widehat b_\tau(s-\tau)}{\tau}\dxs,\\
		&E_\tau^1(t):={}\int_0^{t_\tau}e_\tau^1(s)\ds,\quad E_\tau^2(t):={}\int_0^{t_\tau}e_\tau^2(s)\ds, \\
		&E_\tau^3(t):={}\int_0^{t_\tau}e_\tau^3(s)\ds,\quad	E_\tau^4(t):={}\int_0^{t_\tau}e_\tau^4(s)\ds.
	      \end{align*}
	      Then, \eqref{eqn:discrEI} is equivalent to
	      \begin{align}
		\label{eqn:EIterms}
		  A_\tau(t)+B_\tau(t)+E_\tau^1(t)+E_\tau^2(t) +E_\tau^3(t)+E_\tau^4(t)\leq 0.
	      \end{align}
	      Furthermore, by the a priori estimates, we observe that
	      \begin{align}
		\label{eqn:bddness}
		  |A_\tau(t)|+|B_\tau(t)|+|E_\tau^1(t)|+|E_\tau^2(t)|+E_\tau^3(t)|+|E_\tau^4(t)|<C
	      \end{align}
	      for all $t\in[0,T]$ and for all $\tau>0$ (along a subsequence
	      $\tau_k$). Next, we consider the $\liminf_{\tau\searrow 0}$ of
	      each term in \eqref{eqn:EIterms} separately.
	\begin{itemize}
	  \item By the already proven convergence properties and by lower semi-continuity arguments, we obtain
	    \begin{align}
	      \label{eqn:Aconv}
		\liminf_{\tau\searrow 0}\int_{t_1}^{t_2} A_\tau(t)\dt\geq \int_{t_1}^{t_2} A(t)\dt \text{ for all }0\leq t_1\leq t_2\leq T,
	    \end{align}
	    where $A$ is defined as $A_\tau$ but $c_\tau$, $u_\tau$, $z_\tau$, $v_\tau$ and $\widehat b_\tau$ are substituted by their continuous limits.
	    Note that this $\liminf$--estimate does not necessarily hold
			pointwise a.e. in $t$ because, for instance,
			we do not know $v_\tau(t)\to v(t)$ weakly in $L^2(\Omega;\R^n)$ for a.e. $t$ (see \eqref{eqn:conv5}).
	  \item Let $0\leq t_1\leq t_2\leq T$ be arbitrary. By Fatou's lemma, by \eqref{eqn:conv8} and by a lower semi-continuity argument, we obtain
			\begin{align}
				\liminf_{\tau\searrow0}\int_{t_1}^{t_2}\int_0^{t_\tau}\int_\Omega|\partial_t\widehat z_\tau(s)|^2\dxs\dt
				&\geq\int_{t_1}^{t_2}\bl\liminf_{\tau\searrow 0}\int_0^{t_\tau}\int_\Omega|\partial_t\widehat z_\tau(s)|^2\dxs\br\dt\notag\\
				\label{eqn:Fatou1}
				&\geq \int_{t_1}^{t_2}\int_0^{t}\int_\Omega|\partial_t z(s)|^2\dxs\dt.
	    \end{align}
	    Analogously,
	    \begin{align}
	      \label{eqn:Fatou2}
		\liminf_{\tau\searrow0}\int_{t_1}^{t_2}\int_0^{t_\tau}\int_\Omega
	                   \delta     |\partial_t\widehat c_\tau(s)|^2\dxs\dt
		\geq \int_{t_1}^{t_2}\int_0^{t}\int_\Omega\delta|\partial_t c(s)|^2\dxs\dt
	    \end{align}
	    and, by \eqref{eqn:conv10},
	    \begin{multline}
	      \label{eqn:Fatou3}
		\liminf_{\tau\searrow0}\int_{t_1}^{t_2}\int_0^{t_\tau}\int_\Omega m(c_\tau^-(s),z_\tau^-(s))\nabla\mu_\tau(s)\cdot  \nabla\mu_\tau(s)   \dxs\dt\\
		    \geq \int_{t_1}^{t_2}\int_0^{t}\int_\Omega	m(c(s),z(s))\nabla\mu(s)\cdot  \nabla\mu(s)   \dxs\dt.
	    \end{multline}
	    Taking also \eqref{eqn:bddness} and the already known convergence properties into account, we obtain
	    \begin{align}
	      \label{eqn:Bconv}
		\liminf_{\tau\searrow0}\int_{t_1}^{t_2} B_\tau(t)\dt\geq \int_{t_1}^{t_2} B(t)\dt,
	    \end{align}
	    where $B$ is defined as $B_\tau$ but $c_\tau$, $\widehat c _\tau$,
	      $u_\tau$, $\widehat u_\tau$, $v_\tau^-$, $z_\tau$, $\widehat
	      z_\tau$, $\widehat \mu_\tau$ and $\widehat b_\tau$
	    are substituted by their continuous counterparts and $\frac{\partial_t\widehat b_\tau(t)-\partial_t\widehat b_\tau(t-\tau)}{\tau}$
	    by $\partial_{tt} b(t)$.
	  \item Due to the differentiability of $\mathbf C$ we have
			\begin{align}
			\label{eqn:diff}
				\mathbf C(z_\tau^-) = \mathbf C(z_\tau)+\mathbf C'(z_\tau)(z_\tau^--z_\tau)+r(z_\tau^--z_\tau),\;\frac{r(\eta)}{\eta}\to 0\text{ as }\eta\to 0.
			\end{align}
	    Hence, we obtain
	    \begin{align}
				&\int_0^{t}\int_\Omega\frac12\frac{\mathbf C(z_\tau^-)-\mathbf C(z_\tau)}{\tau}\big( \e(u_\tau^-) - e^*(c) \big) :
					\big( \e(u_\tau^-) - e^*(c) \big)\dxs\notag\\
				&\qquad=\int_0^{t}\int_{\{z_\tau^-(s)\neq z_\tau(s)\}}\frac 12\bl \mathbf C'(z_\tau)\frac{z_\tau^--z_\tau}{\tau}
					+\frac{r(z_\tau^--z_\tau)}{z_\tau^--z_\tau}\frac{z_\tau^--z_\tau}{\tau}\br\big( \e(u_\tau^-) - e^*(c) \big)\notag\\
			\label{eqn:energyTerm}
				&\hspace{5cm}:\big( \e(u_\tau^-) - e^*(c) \big)\dxs
			\end{align}
			Because of
			\begin{align*}
			\begin{split}
				\left\|\frac{r(z_\tau^--z_\tau)}{z_\tau^--z_\tau}\right\|_{L^\infty(\{z_\tau^-\neq z_\tau\})}
				\leq{}\left\|\frac{\mathbf C(z_\tau^-)-\mathbf
	      C(z_\tau)}{z_\tau^--z_\tau}\right\|_{L^\infty(\{z_\tau^-\neq
	      z_\tau\})} \hspace{3.3cm}\\
				\quad+\left\|\mathbf C'(z_\tau)\frac{z_\tau^--z_\tau}{z_\tau^--z_\tau}\right\|_{L^\infty(\{z_\tau^-\neq z_\tau\})}<C,
			\end{split}
			\end{align*}
	    and $\frac{r(z_\tau^--z_\tau)}{|z_\tau^--z_\tau|}\to 0$ a.e. in $\Omega_T$ as $\tau\searrow0$ we conclude
	    by Lebesgue's generalized convergence theorem
	    \begin{align*}
				&\left\|\frac{r(z_\tau^--z_\tau)}{z_\tau^--z_\tau}\right\|_{L^q(\{z_\tau^-\neq z_\tau\})}\to0\text{ for every }q\geq1.
	    \end{align*}
	    Using this and the already known convergence properties, we end up with
	    \begin{align*}
		  	\textit{left hand side of \eqref{eqn:energyTerm}} \to \int_{\Omega_t}W_{,z}(c,\e(u),z)\partial_t z\dxs
	    \end{align*}
	    and, consequently, $E_\tau^1(t)\to0$ as $\tau\searrow0$. Together with the uniform boundedness \eqref{eqn:bddness}, this implies
	    \begin{align}
	      \label{eqn:Econv}
		\int_{t_1}^{t_2} E_\tau^1(t)\dt\to 0\text{ as }\tau\searrow0 \text{ for all }0\leq t_1\leq t_2\leq T.
	    \end{align}
	    The convergence
	    \begin{align}
	      \label{eqn:EEconv}
		\int_{t_1}^{t_2} E_\tau^4(t)\ds\to 0\text{ as }\tau\searrow0\text{ for all }0\leq t_1\leq t_2\leq T,
	    \end{align}
	    can be shown as above.
	  \item
			Noticing the linearity of $e^*$, a short calculation yields
			\begin{align*}
			\begin{split}
				&\int_0^{t_\tau}\int_\Omega\frac{W(c^-_\tau, \e(u_\tau^-),z^-_\tau) - W(c_\tau, \e(u_\tau^-),z^-_\tau)}{\tau}\dx\\
				&\qquad=\int_0^{t_\tau}\int_\Omega \mathbf C(z^-_\tau)\Big(\e(u_\tau^-)-e^*\Big(\frac{c_\tau^-+c_\tau}{2}\Big)\Big):e^*(\partial_t\widehat c_\tau)\dxs.
			\end{split}
			\end{align*}
			Due to the already known convergence properties, we obtain
			\begin{align*}
				\int_0^{t_\tau}\int_{\Omega}\frac{W(c^-,\e(u^-),z^-)-W(c,\e(u^-),z^-) }{\tau}\dxs \to -\int_{\Omega_t}W_{,c}(c,\e(u),z)\partial_t c\dxs
			\end{align*}
			and, consequently, $E_\tau^2(t)\to0$ as $\tau\searrow0$. Together with the uniform boundedness \eqref{eqn:bddness}, this implies
			\begin{align}
				\label{eqn:Econv2}
				\int_{t_1}^{t_2} E_\tau^2(t)\dt\to 0\text{ as }\tau\searrow0 \text{ for all }0\leq t_1\leq t_2\leq T.
			\end{align}
		\item The claim
			\begin{align}
				\label{eqn:EEconv2}
				\liminf_{\tau\searrow0}\int_{t_1}^{t_2} E_\tau^3(t)\dt\geq 0\text{ for all }0\leq t_1\leq t_2\leq T
			\end{align}
			can be shown by the following arguments:
			On the one hand, convexity of $\Psi_1$ yields
			$$
				\frac{\Psi_1(c^-_\tau)-\Psi_1(c_\tau)}{\tau}+\Psi_1'(c_\tau)\partial_t\widehat c_\tau\geq 0.
			$$
			On the other hand, by using the differentiability property of $\Psi_2$, we obtain (cf. \eqref{eqn:diff})
			$$
				\frac{\Psi_2(c^-_\tau) -\Psi_2(c_\tau)}{\tau}+\Psi_2'(c_\tau)\partial_t\widehat c_\tau
				=\frac{r(c_\tau^--c_\tau)}{\tau}
				\text{ with }\frac{r(\eta)}{\eta}\to 0\text{ as }\eta\to 0.
			$$
			In the non-trivial case $c_\tau^--c_\tau\neq 0$, we can argue as follows.
			Since $\frac{r(c_\tau^--c_\tau)}{\tau}=\frac{r(c_\tau^--c_\tau)}{c_\tau^--c_\tau}\frac{c_\tau^--c_\tau}{\tau}$
			and since $\frac{c_\tau^--c_\tau}{\tau}$ is bounded in $L^2(\Omega_T)$,
			it remains to show
			\begin{align}
			\label{eqn:rConv}
				\frac{r(c_\tau^--c_\tau)}{c_\tau^--c_\tau}\to 0\text{ in }L^2(\Omega_T)\text{ as }\tau\searrow0.
			\end{align}
			Indeed, it converges pointwise to $0$ a.e. in $\Omega_T$ and applying the mean value theorem yields
			(here $\xi\in[\min\{c_\tau^-,c_\tau\},\max\{c_\tau^-,c_\tau\}]$)
			\begin{align*}
				\left|\frac{r(c_\tau^--c_\tau)}{c_\tau^--c_\tau}\right|
					&=\left|\frac{\Psi_2(c^-_\tau) -\Psi_2(c_\tau)}{c_\tau^--c_\tau}-\Psi'_2(c_\tau)\right|\\
					&\leq|\Psi_2'(\xi)|+|\Psi'_2(c_\tau)|\\
					&\leq C(1+|\xi|+|c_\tau|)\\
					&\leq C(1+|c_\tau^-|+2|c_\tau|).
			\end{align*}
			Therefore, the left hand side is bounded in $L^\infty(0,T;L^{2^*}(\Omega))$.
			Lebesgue's generalized convergence theorem yields \eqref{eqn:rConv}.
			We end up with \linebreak
			$\liminf_{\tau\searrow0} E_\tau^3(t)\geq 0$ as $\tau\searrow0$.
			Fatou's lemma shows the claim.
	\end{itemize}
	If we combine \eqref{eqn:Aconv}, \eqref{eqn:Bconv}, \eqref{eqn:Econv}, \eqref{eqn:Econv2},  \eqref{eqn:EEconv}
	and \eqref{eqn:EEconv2}, we finally obtain
	\begin{align*}
	  0&\geq \liminf_{\tau\searrow0}\int_{t_1}^{t_2}\bl A_\tau(t)+B_\tau(t)+E_\tau^1(t)+E_\tau^2(t) + E_\tau^3(t)+E_\tau^4(t)\br\dt\\
	   &\geq \int_{t_1}^{t_2}\bl A(t)+B(t)\br\dt.
	\end{align*}
	for all $0\leq t_1\leq t_2\leq T$. Thus, $A(t)+B(t)\leq 0$ for a.e.~$t\in(0,T)$ which is the desired energy inequality \eqref{eqn:weak4}.
	\end{itemize}
	Hence, we obtain existence of weak solutions in the sense of Definition \ref{def:regWeakSolution}.
	\ep

\subsection{Existence proof for the limit system}
	\label{section:limit}

	We now study the limit $\delta\searrow0$. For each $\delta>0$, we obtain a weak solution $(c_\delta, u_\delta, z_\delta, \mu_\delta, \xi_\delta)$
	in the sense of Definition \ref{def:regWeakSolution}.

	\begin{lemma}[\textbf{A priori estimates}]
	\label{lemma:apriori_delta}
	There exists a constant $C>0$ independent of $\delta$ such that
	\begin{itemize}
	  \item[(i)]  $\|c_\delta\|_{L^\infty(0,T;H^1(\Omega))} < C$, \quad
			$\sqrt{\delta}\| \partial_t c_\delta\|_{L^2(0,T;L^2(\Omega))}<C$,\\[1mm]
			$\| \partial_t c_\delta\|_{L^2(0,T;(H^1(\Omega))^*)} < C,$
	  \item[(ii)] $\|u_\delta\|_{L^\infty(0,T;H^1(\Omega;\R^n))\cap W^{1,\infty}(0,T;L^2(\Omega;\R^n))}<C$, \quad
			$\sqrt{\delta}\|u_\delta\|_{L^\infty(0,T;H^2(\Omega;\R^n))}<C$,\\[2mm]
			$\|u_\delta\|_{H^2(0,T;(H_{\Gamma_\mathrm{D}}^2(\Omega;\R^n))^*)}<C$,
		\item[(iii)]  $\|z_\delta\|_{L^\infty(0,T;\Wa)\cap H^1(0,T;L^2(\Omega))}<C$,
		\item[(iv)]   $\|\mu_\delta\|_{L^2(0,T;H^1(\Omega))}<C$, \quad $\|m(c_\delta, z_\delta)^{1/2} \nabla \mu_\delta\|_{L^2(0,T; L^2(\Omega; \R^n))}<C.$
		\end{itemize}
	\end{lemma}
	\begin{proof}
	From the energy inequality \eqref{eqn:weak6_delta}, we infer the second inequality of $(i)$, the first two inequalities of $(ii)$,
	$(iii)$ and the second inequality of $(iv)$.
	By considering \eqref{eqn:weak3_delta}, we get
	\begin{align*}
	  \langle \partial_{tt} u_\delta(t),\zeta\rangle_{H^2}\leq{}
	    & C\Big(\|\e(u_\delta(t))\|_{L^2} +\|c_\delta(t)\|_{L^2}\Big)\|\e(\zeta)\|_{L^2}\\
	    &+\delta\|\nabla\bl\nabla u_\delta(t)\br\|_{L^2}\|\nabla\bl\nabla \zeta\br\|_{L^2}
	    +\|l\|_{L^2}\|\zeta\|_{L^2}
	\end{align*}
	and, therefore,
	\begin{align*}
	  \|u_\delta\|_{H^2(0,T;(H_{\Gamma_\mathrm{D}}^2(\Omega;\R^n))^*)}&<C.
	\end{align*}
	Due to $\int_\Omega c_\delta(t) \dx = const.$ and the boundedness
	of $\| \nabla c_\delta(t) \|_{L^2(\Omega)}$, we derive by Poincar{\'e}'s inequality
	the first inequality of $(i)$.

	From  \eqref{eqn:weak1_delta} and  \eqref{eqn:weak2_delta} we obtain
	boundedness of $\int_\Omega \mu_\delta(t) \dx$. Since
	$\| \nabla \mu_\delta (t) \|_{L^2(\Omega_T)}$ is also bounded,
	Poincar{\'e}'s inequality yields the first inequality of $(iv)$.

	Finally, we know from the boundedness of $\{ \nabla \mu_\delta \}$
	in $L^2(\Omega_T;\R^n)$ that $\{ \partial_t c_\delta \}$ is also bounded in
	$L^2(0,T;(H^1(\Omega))^*) $ with respect to $\delta$ by
	applying equation \eqref{eqn:weak1_delta}. Hence, the third inequality of (i) is satisfied.
	\end{proof}
	\begin{lemma}[\textbf{Convergence properties}]
	There exist functions
	\begin{align*}
	  &c\in L^\infty(0,T;H^1(\Omega))\cap H^1(0,T;(H^1(\Omega))^*),\\
	  &u\in L^\infty(0,T;H^1(\Omega;\R^n))\cap W^{1,\infty}(0,T;L^2(\Omega;\R^n))\cap H^2(0,T;(H_{\Gamma_\mathrm{D}}^2(\Omega;\R^n))^*),\\\
	  &z\in L^\infty(0,T;\Wa)\cap H^1(0,T;L^2(\Omega)),\\
	  &\mu\in L^2(0,T;H^1(\Omega))
	\end{align*}
	and subsequences (omitting the subscript) such that for all $r\geq 1$ and $s<2^*$:
	\begin{align}
		c_{\delta} &\to c&&\text{ weakly-star in }\notag\\
		\label{eqn:conv1a_delta}
		&&&\quad L^\infty(0,T;H^1(\Omega))\cap H^1(0,T;(H^1(\Omega))^*),\\
	  \label{eqn:conv1b_delta}
		&&&\text{ strongly in }L^2(\Omega_T),\text{ a.e. in }\Omega_T,\\
	  \label{eqn:conv3a_delta}
	    u_{\delta} &\to u&&\text{ weakly-star in }L^\infty(0,T;H^1(\Omega;\R^n)),\\
                       &&&\text{ weakly-star in }  \label{eqn:conv3b_delta}
                       	W^{1,\infty}(0,T;L^2(\Omega;\R^n)), \\
	  \label{eqn:conv7a_delta}
	    z_{\delta}&\to z&&\text{ weakly-star in }
\\ && &\quad
L^\infty(0,T;\Wa)\cap H^1(0,T;L^2(\Omega)),\\
	  \label{eqn:conv7b_delta}
			&&&\text{ strongly in }L^r(0,T;L^r(\Omega;\R^n)),\\
			&&&\text{ a.e. in }\Omega_T,\\
	  \label{eqn:conv7c_delta}
			&&&\text{ strongly in } L^p(0,T;W^{1,p}(\Omega;\R^n)),\\
	  \label{eqn:conv7d_delta}
                	&&&\text{ strongly in } C(\overline{\Omega_T}),\\
	  \label{eqn:conv9_delta}
	    \mu_{\delta}&\to \mu&&\text{ weakly in } L^2(0,T;H^1(\Omega)),\\
	   \label{eqn:conv10_delta}
	    m(c_\delta,z_\delta)^{1/2} \nabla \mu_{\delta}&\to m(c,z)^{1/2} \nabla \mu&&\text{ weakly in }L^2(0,T;L^2(\Omega, \R^n))
	\end{align}
	as $\delta\searrow 0$.
        \end{lemma}
        \begin{proof}
	Lemma \ref{lemma:apriori_delta} reveals the existence of functions
	\begin{align*}
	  &c\in L^\infty(0,T;H^1(\Omega))      \\
	  &u\in L^\infty(0,T;H^1(\Omega;\R^n))\cap W^{1,\infty}(0,T;L^2(\Omega;\R^n))\cap H^2(0,T;(H_{\Gamma_\mathrm{D}}^2(\Omega;\R^n))^*),\\
	  &z\in L^\infty(0,T;\Wa)\cap H^1(0,T;L^2(\Omega)),\\
	  &\mu \in L^2(0,T;H^1(\Omega)) , \\
	  & m(c,z)^{1/2} \nabla \mu \in L^2(0,T;L^2(\Omega, \R^n))
	\end{align*}
	and subsequences indexed by $\delta_k$ such that
	\begin{align}
	  c_{\delta_k}&\to c&&\text{ weakly-star in } L^\infty(0,T;H^1(\Omega)),\\
	  u_{\delta_k}&\to u&&\text{ weakly-star in }
\\ && & \hspace{3mm}
L^\infty(0,T;H^1(\Omega;\R^n))\cap W^{1,\infty}(0,T;L^2(\Omega;\R^n)),\\
	  z_{\delta_k}&\to z&&\text{ weakly-star in }L^\infty(0,T;\Wa)\cap H^1(0,T;L^2(\Omega)), \\
	  \mu_{\delta_k}&\to \mu&&\text{ weakly in }L^2(0,T;H^1(\Omega)),\\
	  m(c_{\delta_k},z_{\delta_k})^{1/2} \nabla \mu_{\delta_k}&\to w&&\text{ weakly in } L^2(0,T;L^2(\Omega;\R^n)).
	\end{align}
	Due to the strong convergence properties of $\{c_{\delta_k}\}$,
	$\{z_{\delta_k}\}$ and the growth assumptions on
	the mobility function $m$, we infer
	$$
		w=m(c,z)^{1/2} \nabla \mu.
	$$
	In the following, we omit the subscript $k$.
  Furthermore, property (i) of Lemma \ref{lemma:apriori_delta} shows that $\{c_\delta\}$ converges
        strongly to an element $c$ in $L^2(\Omega_T)$ as $\delta \searrow 0$ for a
        subsequence by a compactness result due to Aubin and Lions
        (\cite{Simon}). By choosing a further subsequence we also obtain
        pointwise almost everywhere convergence. 

	By applying the same technique as for Lemma \ref{lemma:strongConvZ}, strong convergence of $\nabla z_{\delta}$ in $L^p(\Omega_T;\R^n)$
	can be obtained.
	Note that we need the assumption $\mathbf C'(\cdot)\geq 0$, see \eqref{eqn:assumptionC}.
	We conclude that
	\begin{align*}
	  &z_{\delta}\to z&&\text{ strongly in } L^p(0,T;W^{1,p}(\Omega)).\hspace*{15.2em}
	\end{align*}
	Furthermore, by Lemma \ref{lemma:apriori_delta} (iii), we find
	\begin{align*}
	  &z_{\delta}\to z&&\text{ strongly in } \C C(\ol{\Omega_T})\hspace*{20.4em}
	\end{align*}
	for a subsequence by an Aubin-Lions type compactness result (cf. \cite{Simon}).
	\end{proof}

	Next, we will proof our main result.
	
\vspace*{4pt}\noindent{\it{Proof of Theorem \ref{theorem:mainResult}.}}
	\begin{itemize}
	  \item (Cahn-Hilliard equation)\\
            Writing \eqref{eqn:weak1_delta} in the form
	    \begin{align*}
	      \int_{\Omega_T}(c_\delta-c^0)\partial_t\zeta\dxt=\int_{\Omega_T}m(c_\delta,z_\delta)\nabla\mu_\delta\cdot\nabla\zeta\dxt,
	    \end{align*}
	    by only allowing test-functions $\zeta \in L^2(0,T;H^1(\Omega))\cap H^1(0,T;L^2(\Omega))$ with $\zeta(T)=0$ we
	    may pass to the limit by means of the convergence properties \eqref{eqn:conv1b_delta}, \eqref{eqn:conv7c_delta}
	    and \eqref{eqn:conv9_delta} and receive \eqref{eqn:weak1}.

	    Equation \eqref{eqn:weak2} can be obtained by integrating \eqref{eqn:weak2_delta} over time and taking advantage of the
	    convergence properties \eqref{eqn:conv9_delta}, \eqref{eqn:conv1a_delta}, \eqref{eqn:conv1b_delta}, \eqref{eqn:conv3a_delta},
	    \eqref{eqn:conv7c_delta} and Lemma \ref{lemma:apriori_delta} (i).
	 \item (Balance equation of forces)\\
	  Integrating \eqref{eqn:weak3_delta} from $0$ to $T$ and
	    using \eqref{eqn:conv1b_delta}, \eqref{eqn:conv3a_delta}, \eqref{eqn:conv7d_delta} and
	    the convergence
	  $\int_0^T\delta\langle A u_\delta,\zeta\rangle_{H^2}\dt$ $\to 0$ due to
	    Lemma \ref{lemma:apriori_delta} (ii) we conclude
	  \begin{align}
	    \label{eqn:balanceIdentity}
	      \int_0^T\langle \partial_{tt} u,\zeta\rangle_{H^2}\dt+\int_{\Omega_T} W_{,e}(c,\e(u),z):\e(\zeta)\dxt=\int_{\Omega_T} l\cdot\zeta\dxt
	  \end{align}
	  for all $\zeta\in L^\infty(0,T;H_{\Gamma_\mathrm{D}}^2(\Omega;\R^n))$.
		Therefore, \eqref{eqn:weak3} is true for all $\zeta\in H_{\Gamma_\mathrm{D}}^2(\Omega;\break\R^n)$ and a.e. $t\in(0,T)$.
		Using the density of the set $H_{\Gamma_\mathrm{D}}^2(\Omega;\R^n)$ in $H_{\Gamma_\mathrm{D}}^1(\Omega;\R^n)$ (here we need the assumption that the boundary parts
		$\Gamma_\mathrm{D}$ and $\Gamma_\mathrm{N}$ have finitely many path-connected components, see \cite{Ber11}),
		we can identify
		$\partial_{tt}u(t)\in (H_{\Gamma_\mathrm{D}}^1(\Omega;\R^n))^*$
		and \eqref{eqn:weak3} is true for all $\zeta\in H_{\Gamma_\mathrm{D}}^1(\Omega;\R^n)$ and a.e. $t\in(0,T)$.
		Furthermore, $\partial_{tt}u\in L^\infty(0,T;(H_{\Gamma_\mathrm{D}}^1(\Omega;\R^n))^*)$.
	\item (Variational inequality for $z$)\\
	  The variational inequality can be shown as in \cite{HK13_1}. We choose the following cluster points with respect to a subsequence:
	  \begin{align}
	    \chi_\delta:=\chi_{\{z_\delta>0\}}
			&\to\chi&&\text{weakly-star in }L^\infty(\Omega_T),\\
	    \eta_\delta:=\chi_{\{z_\delta=0\}\cap\{W_{,z}(c_\delta, \e(u_\delta),z_\delta)+f'(z_\delta)\leq 0\}}
			&\to\eta&&\text{weakly-star in }L^\infty(\Omega_T),\\
	    \label{eqn:weakConv15}
	    F_\delta:=\chi_{\{z_\delta>0\}}\sqrt{\frac{\mathbf C'(z_\delta)}{2}}(\e(u_\delta)-e^*(c_\delta))&\to F&&\text{weakly in }L^2(\Omega_T;\R^{n\times n}),\\
	    G_\delta:=\chi_{\{z_\delta=0\}\cap\{W_{,z}(c_\delta,\e(u_\delta),z_\delta)+f'(z_\delta)\leq 0\}}\times\\
	    \qquad\times\sqrt{\frac{\mathbf C'(z_\delta)}{2}}(\e(u_\delta)-e^*(c_\delta))
				&\to G&&\text{weakly in }L^2(\Omega_T;\R^{n\times n}).
	  \end{align}
	  Note that since $\mathbf C'(\cdot)$ is symmetric and positive definite matrix, its square root exists.
	  By \eqref{eqn:conv3a_delta} and \eqref{eqn:conv7d_delta}, we obtain for a.e. $x\in\{z>0\}$
	  \begin{align}
	    \label{eqn:limitProp}
	      \chi(x)=1,\quad \eta(x)=0,\quad F(x)=\sqrt{\frac{\mathbf C'(z(x))}{2}}(\e(u)(x)-e^*(c(x))),\quad G(x)=0
	  \end{align}
		because of the following arguments:
		
	  Let $\zeta\in L^2(\Omega_T;\R^{n\times n})$ with $\mathrm{supp}(\zeta)\subseteq\{z>0\}$.
		Then, by \eqref{eqn:conv7d_delta}, we obtain $\mathrm{supp}(\zeta)\subseteq\{z_\delta>0\}$ for all sufficiently small $\delta>0$.
		By \eqref{eqn:weakConv15}, we find
		$$\int_{\Omega_T}F_\delta:\zeta\dxt\to\int_{\Omega_T}F:\zeta\dxt.$$
		On the other hand, by \eqref{eqn:conv3a_delta}, (note that $\delta$ can be chosen arbitrarily small)
		\begin{align*}
			\int_{\Omega_T}F_\delta:\zeta\dxt&=\int_{\Omega_T}\sqrt{\frac{\mathbf C'(z_\delta)}{2}}(\e(u_\delta)-e^*(c_\delta)):\zeta\dxt\\
			&\to\int_{\Omega_T}\sqrt{\frac{\mathbf C'(z)}{2}}(\e(u)-e^*(c)):\zeta\dxt
		\end{align*}
		Thus,
		$$
			\int_{\Omega_T}\sqrt{\frac{\mathbf C'(z)}{2}}(\e(u)-e^*(c)):\zeta\dxt=\int_{\Omega_T}F:\zeta\dxt
		$$
		The other identities in \eqref{eqn:limitProp} follow analogously.
	
	Now let $\zeta\in L^\infty(0,T;W_-^{1,p}(\Omega))$.
	Taking \eqref{eqn:xiDef} into account, inequality \eqref{eqn:weak2} becomes by integration over time
	\begin{align*}
		0\leq{}&\int_{\Omega_T}\bl|\nabla z_\delta|^{p-2}\nabla z_\delta\cdot\nabla\zeta+\partial_t z_\delta\zeta\br\dxt\\
			&+\int_{\{z_\delta>0\}}\bl W_{,z}(c_\delta,\e(u_\delta),z_\delta)+f'(z_\delta)\br\zeta\dxt\\
			&+\int_{\{z_\delta=0\}\cap \{W_{,z}(c_\delta,\e(u_\delta),z_\delta)+f'(z_\delta)\leq 0\}}\bl
			W_{,z}(c_\delta,\e(u_\delta),z_\delta)+f'(z_\delta)\br\zeta\dxt.
	\end{align*}
	Applying $\limsup_{\delta\searrow0}$ on both sides and multiplying by $-1$ yield
	\begin{align*}
	    0\geq{}&\lim_{\delta\searrow0}\int_{\Omega_T}\bl|\nabla z_\delta|^{p-2}\nabla z_\delta\cdot\nabla(-\zeta)+\partial_t z_\delta(-\zeta)\br\dxt\\
		    &+\liminf_{\delta\searrow0}\int_{\Omega_T}(F_\delta)^2(-\zeta)\dxt+\liminf_{\delta\searrow0}\int_{\Omega_T}\chi_\delta\,f'(z_\delta)(-\zeta)\dxt\\		    &+\liminf_{\delta\searrow0}\int_{\Omega_T}(G_\delta)^2(-\zeta)\dxt+\liminf_{\delta\searrow0}\int_{\Omega_T}\eta_\delta\,f'(z_\delta)	 (-\zeta)\dxt.
	\end{align*}
	Weakly lower semicontinuous arguments, the uniformly convergence property \eqref{eqn:conv7d_delta} and the properties listed in \eqref{eqn:limitProp} give
	\begin{align*}
	  0\geq{}&\int_{\Omega_T}\bl|\nabla z|^{p-2}\nabla z\cdot\nabla(-\zeta)+\partial_t z(-\zeta)\br\dxt\\
		 &+\int_{\{z>0\}}\bl W_{,z}(c,\e(u),z)+f'(z)\br(-\zeta)\dxt\\
		 &+\int_{\{z=0\}}\bl (F^2+G^2)+(\chi+\eta) f'(z)\br(-\zeta)\dxt.
	\end{align*}
	This inequality may also be written in the following form:
	\begin{align*}
	  0\leq{}&\int_{\Omega_T}\bl|\nabla z|^{p-2}\nabla z\cdot\nabla\zeta+\bl W_{,z}(c,\e(u),z)+f'(z)+\partial_t z\br\zeta\br\dxt\\
		 &+\int_{\{z=0\}}\bl (F^2+G^2)+(\chi+\eta) f'(z)-W_{,z}(c,\e(u),z)-f'(z)\br\zeta\dxt.
	\end{align*}
	Therefore,
	\begin{align*}
	  0\leq{}&\int_{\Omega_T}\bl|\nabla z|^{p-2}\nabla z\cdot\nabla\zeta+\bl W_{,z}(c,\e(u),z)+f'(z)+\partial_t z+\xi\br\zeta\br\dxt
	\end{align*}
	with
	$$\xi:=\chi_{\{z=0\}}\mathrm{min}\Big\{0,(F^2+G^2)+(\chi+\eta-1) f'(z)-W_{,z}(c,\e(u),z)\Big\}.$$
	This proves \eqref{eqn:weak4} and \eqref{eqn:weak5}.
			
	\item (Energy inequality)\\
	  To prove the energy inequality \eqref{eqn:weak6}, we
	  can proceed as in the proof of Theorem \ref{theorem:mainResult_delta}.
	  Integrating \eqref{eqn:weak6_delta} with respect to time on $[t_1,t_2]$ yields ($0\leq t_1\leq t_2\leq T$)
	  \begin{align}
	    \label{eqn:EIlimit}
		&\int_{t_1}^{t_2} \bl A_\delta(t)+B_\delta(t)+C_\delta(t)\br\dt\leq 0
	  \end{align}
	  with
	  \begin{align*}
	    &A_\delta(t):=\\
	    &\quad\int_\Omega\Big(\frac 1p|\nabla
			  z_\delta(t)|^p+ \frac 12|\nabla c_\delta(t)|^2+W(c_\delta, \e(u_\delta(t)),z_\delta(t))
				+f(z_\delta(t)) +\Psi(c_\delta(t))\Big)\dx\\
				&\quad-\int_\Omega\bl\frac 1p|\nabla z^0|^p+\frac 12|\nabla c^0|^2
				+W(c^0,\e(u^0),z^0)+f(z^0)+ \Psi(c^0)\br\dx\\
				&\quad+\int_\Omega\frac 12|\partial_t u_\delta(t)|^2\dx-\int_\Omega\frac 12|v^0|^2\dx
				-\int_\Omega \partial_t u_\delta(t)\cdot \partial_t b(t)\dx+\int_\Omega v^0\cdot \partial_t b^0\dx,\\
			&B_\delta(t):=
				\int_{\Omega_t} \big(|\partial_t z_\delta|^2 + \delta|\partial_t c_\delta|^2
			      + m(c_\delta,z_\delta) \nabla \mu_\delta \cdot \nabla \mu_\delta \big) \dxs \\
			   &\qquad\qquad-\int_{\Omega_t} W_{,e}(c_\delta, \e(u_\delta),z_\delta):\e(\partial_t b)\dxs
			    +\int_{\Omega_t}\partial_{t}u_\delta\cdot\partial_{tt} b\dxs\\
			    &\qquad\qquad-\int_{\Omega_t}l\cdot (\partial_t u_\delta-\partial_t b)\dxs,\\
	    &C_\delta(t):=\frac\delta2\langle A u_\delta(t),u_\delta(t)\rangle_{H^2}-\frac\delta2\langle Au^0,u^0\rangle_{H^2}
			      -\delta\int_0^t \langle Au_\delta(t),\partial_t b(t)\rangle_{H^2}\dt.
	\end{align*}
	Let $A$ be the corresponding integral expression to $A_\delta$, where $c_\delta$, $u_\delta$ and $z_\delta$
	are replaced by $c$, $u$ and $z$, respectively.
	Furthermore, let
	\begin{align*}
		B(t):={}&\int_{\Omega_t} \big(|\partial_t   z|^2 + m(c,z ) \nabla \mu  \cdot \nabla \mu  \big) \dxs
		-\int_{\Omega_t} W_{,e}(c,\e(u ),z ):\e(\partial_t b)\dxs\\
		& +\int_{\Omega_t}\partial_{t}u \cdot\partial_{tt} b\dxs-\int_{\Omega_t}l\cdot (\partial_t u -\partial_t b)\dxs.
	\end{align*}
	The limit passage in \eqref{eqn:EIlimit} can be performed as follows.
	\begin{itemize}
	  \item Weakly lower semi-continuity arguments show
	    $$\liminf_{\delta\searrow0}\int_{t_1}^{t_2} A_\delta(t)\dt\geq\int_{t_1}^{t_2} A(t)\dt.$$
	  \item Fatou's lemma and weakly lower semicontinuous arguments for $\nabla \mu_\delta$ as
	  well as the convergence property for $c_\delta$,  $u_\delta$, $z_\delta$
		(see  \eqref{eqn:conv1a_delta},  \eqref{eqn:conv1b_delta},   \eqref{eqn:conv3a_delta}, \eqref{eqn:conv7a_delta}, \eqref{eqn:conv7b_delta}) show (cf.~\eqref{eqn:Fatou1}-\eqref{eqn:Fatou3})
	    $$\liminf_{\delta\searrow0}\int_{t_1}^{t_2} B_\delta(t)\dt\geq\int_{t_1}^{t_2} B(t)\dt.$$
	  \item We have
	    $$ C_\delta(t)\geq -\frac\delta2\langle Au^0,u^0\rangle_{H^2}
		-\delta\|u_\delta(t)\|_{H^2(\Omega;\R^n)}\|\partial_t b(t)\|_{H^2(\Omega;\R^n)}. $$
	    By Lemma \ref{lemma:apriori_delta} (ii), we obtain
	    $$ \liminf_{\delta\searrow0}\int_{t_1}^{t_2} C_\delta(t)\dt\geq 0.$$
	\end{itemize}
	We end up with $\int_{t_1}^{t_2} A(t)+B(t)\dt\leq 0$ for all $0\leq t_1\leq t_2\leq T$.
	This proves \eqref{eqn:weak6}.
	\end{itemize}
	
	Putting all steps together, Theorem \ref{theorem:mainResult} is proven.\ep
                        
\begin{scriptsize} 

\bibliographystyle{alpha}
\bibliography{references} 

\end{scriptsize}

\end{document}